\documentclass[a4paper,10pt]{amsart}
\usepackage{amssymb}
\usepackage{cmap}
\usepackage{tikz-cd}
\usepackage[pdfdisplaydoctitle=true,
            colorlinks=true,
            urlcolor=blue,
            citecolor=blue,
            linkcolor=blue,
            pdfstartview=FitH,
            pdfpagemode= UseNone,
            bookmarksnumbered=true,
            backref=page]{hyperref}
\usepackage{hyperxmp}
\usepackage{todonotes}
\usepackage[all]{xy}
\usepackage{ulem}
\usepackage{cancel}

\theoremstyle{plain}
\newtheorem{thm}{Theorem}[subsection]
\newtheorem*{thm*}{Theorem}

\newtheorem{cor}[thm]{Corollary}
\newtheorem{lem}[thm]{Lemma}
\newtheorem{prop}[thm]{Proposition}

\theoremstyle{definition}

\theoremstyle{definition}
\newtheorem{rem}[thm]{Remark}
\newtheorem{defi}[thm]{Definition}

\newtheorem{expl}[thm]{Example}
\newtheorem*{ass*}{Assumption}

\numberwithin{equation}{section}

\newcommand{\RR}{\mathbb{R}} %
\newcommand{\CC}{\mathbb{C}} %

\newcommand{\id}{\mathrm{id}}           %
\let\on=\operatorname

\newcommand{\Diff}{\mathrm{Diff}}       %
\newcommand{\biLip}{\mathrm{biLip}}

\let\on=\operatorname

\def\Diff{\on{Diff}}

\allowdisplaybreaks

\usepackage{todonotes}
\hypersetup{
    pdftitle={Sobolev metrics on spaces of manifold valued curves}, %
    pdfsubject={MSC 2010: 58B20, 58D10, 35G55, 35G60}, %
    pdfkeywords={Sobolev metrics, manifold valued curves}, %
    pdflang=EN %
    }
\author{Yuxiu Lu }
\address{Yuxiu Lu: Department of Mathematics and Statistics, Nanjing University of Science and Technology}
\email{luyxnj@gmail.com }
\date{December 15, 2025}

\begin{document}

\title[The $L^p$-geometry  and its applications]{The $L^p$-geometry and its applications}

%\date{\today}
%
%

%
\begin{abstract}
We generalize the classical Fisher information metric on statistical models to $L^p$-metrics on various spaces of differential forms or group of diffeomorphisms. Using this new interpretation from information geometry, we derive several new results in geometry on group of diffeomorphisms, symplectic geometry and Teichm\"{u}ller theory. This includes geometry of $\operatorname{Diff}_{-\infty}(\RR)$, similar to that of universal Teichm\"{u}ller space in essence, also a study on the space of all symplectic forms on a symplectic manifold $M$ and a generalization of Gelfand-Fuchs cocycles to higher-dimension.
Furthermore, we answer questions in $\alpha$-geometry posed by Gibilisco, and generalize the $L^p$-metrics to geometry on Orlicz spaces. 
\end{abstract}

\maketitle

\setcounter{tocdepth}{2}
\tableofcontents

\section{Introduction}
The classical information geometry considered a parametrized statistical model as a Riemannian manifold endowed with 
the Fisher information metric $g$, 
\begin{align*}
g_{ij}(\theta)=\int_X \frac{\partial \operatorname{log}p(x;\theta)}{\partial\theta_i}\frac{\partial \operatorname{log}p(x;\theta)}{\partial\theta_j}p(x;\theta)dx.     
\end{align*}
This Riemannian metric is central to information geometry in many aspects, as shown by the \u{C}encov theorem and Cram\'er-Rao inequality; see e.g.,~\cite{Amari2000, Ay2017,rao1993,cramer1946}. There are several decent approaches to generalize this structure between differential geometry and statistics. One is to study the whole space $\operatorname{Prob}(M)$ of all probability densities on a compact manifold $M$. It turns out that by using Moser's trick~\cite{moser1965}, $\operatorname{Prob}(M)$ is diffeomorphic to the quotient $\operatorname{Diff}(M)/\operatorname{Diff}_\mu(M)$ of a smooth ILH principal bundle $\operatorname{Diff}(M)$ with the fibre $\operatorname{Diff}_\mu(M)$; for ILH structure, see~\cite{omori1970group}. In their paper~\cite{khesin2011}, Khesin et. al. showed that
the degenerate right-invariant $\dot H^1$ Riemannian metrics on
the full diffeomorphism group $\operatorname{Diff}(M)$,
\begin{align*}
\langle u, v\rangle_{\dot H^1}=\frac{1}{4}\int \operatorname{div}u\cdot\operatorname{div}v\ \mu,    
\end{align*}
descends to a non-degenerate Riemannian metric on the homogeneous space
of right cosets  $\operatorname{Prob}(M)\cong\operatorname{Diff}(M)/\operatorname{Diff}_\mu(M)$. Therefore, a natural object from statistics and probability can be investigated from a infinite differential-geometric and hydrodynamical perspective. 

The geometric approach to 
hydrodynamics dates back to Arnold~\cite{arnold1966geometrie}, in which he observed that the motion of a fluid in a compact  Riemannian manifold $M$ can be interpreted as  a geodesic in the infinite dimensional group $\operatorname{Diff}(M)$ of volume-preserving diffeomorphisms of $M$. This framework of hydrodynamics include many partial differential equations(PDEs)  of interest in 
mathematical physics, known as Euler-Poincar\'e or Euler-Arnold equations. The PDEs discussed in this paper are two classes related to the projective structure on $M$: $r$-Hunter Saxton equations or $\alpha$-Proudman-Johnson equation
(correlations presented in~\cite{bauer2021paper}) interpreted as as geodesic equations of right invariant homogeneous
$\dot W^{1,r}$-Finsler metrics on the diffeomorphism group~\cite{cotter2020r}; Korteweg-de Vries(KdV) equation, interpreted as geodesic equation of $L^2$-metric on the  Virasoro group~\cite{khesin1987}.

Another way to generalize is to replace the inner product of a Riemannian metric by $L^p$-norms of a Finsler structure. 
Bauer et.~al. introduced the $L^p$-Fisher-Rao metric and investigated its geometrical aspects, e.g., geodesic equations/flows, connections and curvatures in the recent paper~\cite{bauer2023}. This treatment has far more potential to explore, the space of probability densities $\operatorname{Prob}(M)$ or positive densities $\operatorname{Dens}(M)$ can be viewed as a subspace of highest dimensional differential forms $\Omega^n(M)$. It is then natural to investigate lower dimensional differential forms endowed with invariant $L^p$-metrics. One of the most important such spaces is the set of all symplectic forms on a symplectic manifold. In fact,  this idea is rooted in Moser's original paper~\cite{moser1965}, as it puts the volume forms and symplectic forms together for study. Furthermore, we can extend some  considerations of the $L^p$-geometry to a more general function space called Orlicz space. This idea was first proposed by Gibilisco~\cite{Gib20}. 

Finally note that all our discussions on the unparametrized space of probability densities $\operatorname{Prob}(M)$ are based on the compactness of a manifold $M$ since otherwise the Fisher information metric and other generalizations easily diverge, and Moser's trick fails as well. However, the classical information geometry is created for a differential-geometric approach to probability densities on a real line $\RR$, and the single focus on probability with compact supports will exclude many more interesting distributions, e.g., the normal distribution. A comprehensive study on diffeomorphism groups of Euclidean spaces $\RR^n$ was firstly conducted by Michor and Mumford in their papers~\cite{michor2013zoo}, where they investigated classes of Fr\'echet diffeomorphism group on $\RR$ with different decay conditions. This work was developed into a study on geometric and hydrodynamic aspects of various diffeomorphism groups of $\RR$ by Bauer et. al. in~\cite{bauer2014homogeneous}. In their recent work~\cite{bauer2021paper}, Bauer. et. al. further spotted a convenient Fr\'echet Lie group $\operatorname{Diff}_{-\infty}(\RR)$ to study the right invariant homogeneous $\dot W^{1,r}$-Finsler metric, and corresponding geodesic equation, i.e., $r$-Hunter Saxton equations. This hydrodynamic consideration generates
an unexpected consequence:
a $L^p$-approximation to a diffeomorphic isometry from $\operatorname{Diff}_{-\infty}(\RR)$ to the function space $W^{\infty,1}$; see corollary~\ref{inftyisom}.
This mapping turns out to be a sort of potential function for the Schwarzian derivative, which leads to a real counterpart of Teichm\"{u}ller theory. Along the whole Section~\ref{caseRsec}, we will explore the similarity between   $\operatorname{Diff}_{-\infty}(\RR)$ and  universal Teichm\"{u}ller space. The study on geometry and hydrodynamics of universal Teichm\"{u}ller space pioneered by many previous researchers; see e.g.,~\cite{ratiu2015,preston2016}.

\subsection*{The structure of the paper}
Section~\ref{lpfishersec}
gives some background information of the $L^p$-Fisher-Rao metric and its geometry. Section~\ref{caseRsec} considers diffeomorphism groups of the non-compact case $\RR$: we find isometries between diffeomorphism groups, function spaces and space of densities, and use them to derive a real counterpart of Teichm\"{u}ller theory. Section~\ref{symplecticsec} concerns the ILH principal bundle structure on the space  of all symplectic structures on a symplectic manifold, we further investigate the invariant metric and local property of this infinite dimensional manifold. Section~\ref{gelfandsec}
is about generalizing the Gelfand-Fuchs $2$-cocycle to higher-dimensional cocycles on $\operatorname{Prob}(M)$.
The final Section~\ref{orliczsec}
is a generalization of the $L^p$-norms to Orlicz spaces with Luxemburg norm.
\subsection*{Acknowledgements}
This work is based on the author's PhD thesis in FSU.
The author is grateful to his advisor Martin Bauer, colleagues Alice~Le~Brigant, Cy~Maor and Hui~Sun. Section~\ref{caseRsec} is motivated by a talk with Koji~Fujiwara in the 10-th China-Japan Geometry Conference hosted by NJUST. The author would like to thank him and the local organizers.

\section{$L^p$-Fisher-Rao metric and its geometric properties}\label{lpfishersec}
\subsection{The space of densities as a Fr\'echet manifold}
Throughout this paper, we work on a closed oriented manifold $M$ of dimension $n$ in smooth category. The only exception is Section~\ref{caseRsec}, in
which we discuss the non-compact case $\RR$. Two important spaces haunt the whole paper are the space of positive densities $\operatorname{Dens}(M)$ and its subspace of all probability densities $\operatorname{Prob}(M)$  
\begin{align*}
\operatorname{Dens}(M):&=\left\{\mu\in\Omega^n(M): \mu>0\right\} \\
\operatorname{Prob}(M):&=\left\{\mu\in\Omega^n(M): \mu>0, \; \int_M\mu=1\right\}. 
\end{align*}
Here, $\Omega^n(M)$ is the space of all $n$-forms on $M$, and the positivity means the positivity of the Radon-Nikodym derivative $\frac{\mu}{\nu}$
 of $\mu$ with respect to a fixed volume form $\nu$ on $M$. In general, the space of $k$-form $\Omega^k(M)$ carries a Fr\'echet manifold structure. Therefore, as a open subset of $\Omega^n(M)$, the space $\operatorname{Dens}(M)$ is also a Fr\'echet manifold with the tangent space $T_\mu \operatorname{Dens}(M)=\Omega^n(M)$. In the same manner, as a convex subset of $\operatorname{Dens}(M)$, $\operatorname{Prob}(M)$ is a Fr\'echet submanifold with the tangent space given by
 \begin{align*}
 T_\mu\operatorname{Prob}(M):&=\left\{a\in\Omega^n(M):  \int_M a=0\right\}.    
 \end{align*}
We consider the pushforward of a volume form $\mu$ by a diffeomorphism $\varphi$,
 \begin{align*}
\varphi_*\mu=(\varphi^{-1})^*\mu=\operatorname{Jac}_\mu(\varphi^{-1})\mu.     
 \end{align*}
By using Moser's trick, the action of $\operatorname{Diff}(M)$ on $\operatorname{Prob}(M)$ is transitive, and the kernel is exactly the infinite dimensional group $\operatorname{Diff}_\mu(M)$ of volume-preserving diffeomorphisms(volumorphisms) of $M$, thus we have
\begin{equation}
\operatorname{Prob}(M)\cong\operatorname{Diff}(M)/\operatorname{Diff}_\mu(M),
\end{equation}
where $\operatorname{Diff}_\mu(M)$ acts by left cosets.
The equality above is essential as it bridges two totally different parts, the space of densities and group of diffeomorphisms. 

\subsection{The $L^p$-Fisher-Rao metric and its geodesic equation}  
The fundamental Finsler metric used in this paper is the $L^p$-Fisher metric, a generalization of the Fisher-Rao metric that coincides with the original one for the case $p=2$.
\begin{defi}
The $L^p$-Fisher metric is defined as a Finsler metric on $\operatorname{Dens}(M)$,
\begin{align}\label{Finsler Fisher metric}
F({\mu},a):=\left(\int|\frac{a}{\mu}|^p\mu\right)^{1/p},
\end{align}
 where $a\in T_\mu\operatorname{Dens}(M)$ denotes the tangent vector at the point $\mu$.
\end{defi}
\begin{rem}
 Note that by using Moser's trick, the $L^p$-Fisher-Rao metric is equivalent to the right invariant $\dot W^{1,p}$ metric on $\operatorname{Diff}(M)/\operatorname{Diff}_\mu(M)$
\begin{align}\label{sobolevmetricdiff}
||u||_{\dot W^{1,p}}=\frac{1}{p^p}\int|\operatorname{div}u|^p \ \mu,
\end{align}
where $u\in T_{e}\operatorname{Diff}(M)$ and $\eta\in\operatorname{Diff}(M)$.
This interpretation of Fisher-Rao metric was originally introduced by Khesin et. al. in Section 3 of~\cite{khesin2011}.   
\end{rem}

In the following theorem we calculate the geodesic equation of the $L^p$-Fisher-Rao metric on both the space of positive densities $\operatorname{Dens}(M)$ and probability densities $\operatorname{Prob}(M)$:
\begin{thm}\label{lpfishereq}
For any $p\in(1,\infty)$, the geodesic equation of the $L^p$-Fisher-Rao metric on the space of densities $\operatorname{Dens}(M)$ is given by
\begin{align*} 
\partial_t(\frac{\rho_t}{\rho})+\frac{1}{p}(\frac{\rho_t}{\rho})^2=0,
\end{align*}
While the geodesic equation of the $L^p$-Fisher-Rao metric on the space of probability densities $\operatorname{Prob}(M)$ is given by
\begin{align}\label{lpgeodesicequ} 
\nabla\left(|\frac{\rho_t}{\rho}|^{p-2}\partial_t(\frac{\rho_t}{\rho})+\frac{1}{p}|\frac{\rho_t}{\rho}|^p\right)=0.
\end{align}
Note that here we identify the path $\mu(x,t)=\rho(x,t) dx$ on $\operatorname{Dens}(M)$ with a path $\rho=\rho(x,t)$ on $C^\infty(M)$ by using the Poincar\'e's duality in the smooth category.

%In particular, the geodesic equation of the Fisher-Rao metric on the space of densities $\operatorname{Dens}(M)$ is given by\begin{align*} 2\frac{d}{dt}(\frac{\varphi_t}{\varphi})+(\frac{\varphi_t}{\varphi})^2=0.\end{align*}
\end{thm}
\begin{rem}
From the above theorem, we can see that the geodesic equation of the $L^p$-Fisher-Rao metric on $\operatorname{Dens}(M)$ coincides with the geodesic equation of the $\alpha$-connection on $\operatorname{Dens}(M)$ but not on $\operatorname{Prob}(M)$. For $M=S^1$, their geodesic equations will lead to two different PDEs: one is the Hunter-Saxton equation which is derived from the Chern connection; the other is the Proudman-Johnson equation which is derived from the $\alpha$-connection; see e.g.,~\cite{bauer2021paper,bauer2023}. 
\end{rem}

\begin{proof}
The length functional of a Finsler-metric $F$ on $\operatorname{Dens}(M)$ is defined as
\begin{equation}\label{length}
L(\rho)=\int_0^1 F(\rho dx,\rho_t dx) dt,   
\end{equation}
where $\rho: [0,1]\to C^\infty(M)$ is a path such that $\mu=\rho dx\in \operatorname{Dens}(M)$ and where $\rho_t$ denotes its derivative. 
A geodesic is a path that locally minimize the length functional; since $L$ is invariant to reparametrization, we can restrict ourselves to paths of constant speed.
By H\"older inequality, it is immediate that constant speed geodesics are exactly the local minimizers of the $q$-energy 
\begin{equation}\label{energyfunctionlp}
E_q(\rho)=\int_0^1 F^q(\rho dx,\rho_t dx) dt,   
\end{equation}
for any $q> 1$.
In our case, for the $L^p$-Fisher-Rao metric the most convenient choice is to consider the $q$-Energy with $q=p$. 
The corresponding energy functional reads as
\begin{equation} 
E_p(\rho)=\frac{1}{p}\int_0 ^1 \int_M|\frac{\rho_{t}}{\rho}|^p \rho dxdt,\end{equation}
Then we can calculate the variation of the $p$-energy functional by means of integral by parts
\begin{align*}
\delta E_p(\rho)(\delta \rho)&=\frac{1}{p}\int_0 ^1 \int  p|\frac{\rho_t}{\rho}|^{p-2}\frac{\rho_t}{\rho}\delta \rho_t-(p-1)|\frac{\rho_t}{\rho}|^p  \delta \rho  \ dxdt\\ 
&=-\frac{1}{p}\int_0 ^1 \int  \left( p\partial_t(|\frac{\rho_{t}}{\rho}|^{p-2}\frac{\rho_t}{\rho})+(p-1)|\frac{\rho_t}{\rho}|^p \right) \  \delta \rho  \ dxdt,
\end{align*}
Thus we can read off the geodesic equation
\begin{align*} 
p\partial_t(|\frac{\rho_{t}}{\rho}|^{p-2}\frac{\rho_t}{\rho})+(p-1)|\frac{\rho_t}{\rho}|^p=0.
\end{align*}
We can simplify to the following form
\begin{equation} 
\partial_t|\frac{\rho_t}{\rho}|+\frac{1}{p}(\frac{\rho_t}{\rho})^2=0
\end{equation}
By applying the same method above to the $L^p$-Fisher-Rao metric on $\operatorname{Prob}(M)$, we have the following energy functional
\begin{align*} 
\tilde E_p(\rho)=\frac{1}{p}\int_0 ^1 \int|\frac{\rho_{t}}{\rho}|^p \rho dxdt-\lambda\int\left(\int\rho dx-1\right)dt,
\end{align*}
where $\lambda$ is a constant with respect to the variable $x$.
Then the variation of the $p$-energy functional is given by
\begin{align*}
\delta\tilde E_p(\rho)(\delta \rho)=-\frac{1}{p}\int_0 ^1 \int  \left( p\partial_t(|\frac{\rho_{t}}{\rho}|^{p-2}\frac{\rho_t}{\rho})+(p-1)|\frac{\rho_t}{\rho}|^p-\lambda\right) \  \delta \rho  \ dxdt,
\end{align*}
After applying the gradient, the geodesic equation reads
\begin{equation}
\nabla\left(|\frac{\rho_t}{\rho}|^{p-2}\partial_t(\frac{\rho_t}{\rho})+\frac{1}{p}|\frac{\rho_t}{\rho}|^p\right)=0.
\end{equation}
\end{proof}
In fact, we have a more explicit form of the geodesic equation on $\operatorname{Prob}(M)$, which can be derived from the Chern connection; see ~\cite{Lu2023}: 
\begin{align*}
 \partial_t(\frac{\rho_t}{\rho})+\frac{1}{p}|\frac{\rho_t}{\rho}|^2+\frac{p-1}{p}\left(\int|\frac{\rho_t}{\rho}|^2\mu/\int|\frac{\rho_t}{\rho}|^{2-p}\mu\right)|\frac{\rho_t}{\rho}|^{2-p}=0,  
 \end{align*}
 where $\rho(t)=\mu(t)/dx$ is a curve of probability functions.
 Note that the above formula for the geodesic equation has to be understood formally only.  Since the tangent vector $\rho_t$ must have some zeros on $\operatorname{Prob}(M)$, the quantity $|\frac{\rho_t}{\rho}|^{2-p}$ in the equation is not well-defined and might not be integrable for $p>2$. This gives a negative answer to the complete integrability for the $\alpha$-Proudman-Johnson equations.

As in the case of Fisher-Rao metric~\cite{lenells2007hunter,lenells2008hunter}, we have an isometry from the space of densities with the $L^p$-Fisher-Rao metric to a flat space.
\begin{thm}\label{lpfisherisometry}
The mapping 
\begin{align*}
\Phi: \begin{cases}
\left(\operatorname{Dens}(M), F\right)  &\to  \left(C^{\infty}(M),L^p\right)\\ \mu &\mapsto p\left(\frac{\mu}{dx}\right)^{1/p}
\end{cases}
\end{align*}
is an isometric embedding.
Furthermore, the image $\mathcal U=\Phi(\operatorname{Dens}(M))$  is 
the set of all positive functions in $C^{\infty}(M)$, i.e.,
\begin{align*}
\mathcal U=\{f\in C^{\infty}(M):\ f>0\}.
\end{align*}
In particular, the restriction map $\Phi|_{\operatorname{Prob}(M)}$ is also an isometric embedding, and the image $\mathcal U'=\Phi(\operatorname{Prob}(M))$  is 
the set of all positive functions in a $L^p$-sphere, i.e.,
\begin{align*}
\mathcal U'=\{f\in C^{\infty}(M):\ f>0,\ ||f||_{L^p}=p\}.
\end{align*}
\end{thm}
\begin{proof}
Fix a point $\mu$, then the differential of $\Phi$ or $\Psi$ at this point along the vector $a$ is 
 \begin{align*}
D_\mu\Phi(a)=\frac{a}{dx}(\frac{\mu}{dx})^{1/p-1}.
\end{align*}
Therefore the norm induced by the embedding $\Phi$ is given by 
 \begin{align*}
|| D_\mu\Phi(a)||_{L^p}^p=\int|D_\mu\Phi(a)|^{p}dx=\int|\frac{a}{\mu}|^p\mu,
\end{align*}
which implies that the embedding $\Phi$ is an isometry preserving the $L^p$-Fisher-Rao metric. 
 \end{proof}
By using this isometry, we can give a explicit expression of the geodesic for the $L^p$-Fisher-Rao metric on 
$\operatorname{Dens}(M)$.
\begin{thm}\label{geodLPFR}
Given boundary conditions $\rho(0)=\rho_0, \rho(1)=\rho_1$ such that $\rho_0 dx, \\ \rho_1 dx\in\operatorname{Dens}(M)$, 
the unique geodesic of the $L^p$-Fisher-Rao metric on $\operatorname{Dens}(M)$ that connects $\rho_0 dx$ to $\rho_1 dx$ is given by
 \begin{align*}
\rho(t)=(t\sqrt[p]{\rho_1}+(1-t)\sqrt[p]{\rho_0})^p, \ \ 0\leq t\leq1,
\end{align*}
where we identify the probability function $\rho$ with the probability density $\rho dx$ as above.
\end{thm}
 \begin{proof} 
It is clear that the geodesic on a flat space is exactly a straight line. This is easy to see since the geodesic connecting two positive functions $f$ and $g$ within $\mathcal U$ is the line $\alpha(t,x)=tf(x)+(1-t)g(x)$, $0\leq t\leq 1$. 
So by using the pullback of the isometry $\Phi$, we have the following geodesic equation on $\operatorname{Dens}(M)$ 
 \begin{align*}
\Phi(\rho(t))=t\Phi(\rho_1)+(1-t)\Phi(\rho_0), \ \ 0\leq t\leq1.
\end{align*}
 \end{proof}
It remains an open question whether it is possible to find the explicit formula of geodesics on $\operatorname{Prob}(M)$: through the isometry given by Theorem~\ref{lpfisherisometry}, it suffices to find the explicit solution of geodesics on a $L^p$-sphere, but this seems not be possible in closed form. We hope that this viewpoint will still provide further tools to study the geodesics on $\operatorname{Prob}(M)$.

\subsection{The \u{C}encov theorem}
The \u{C}encov theorem is fundamental in information geometry. It states that the Fisher-Rao metric on statistical models is essentially the only Riemannian metric that is invariant under sufficient statistics. 
The original version of \u{C}encov theorem~\cite{campbell1986, chentsov1978, cencov1982} only applies to the restricted setting of statistical models with finite data spaces, i.e., it states that the Fisher information metric is the only metric (up to a multiplicative constant) that is defined on all models with finite data spaces and is invariant under all sufficient statistics. For smooth category, the following version of \u{C}encov theorem has been shown for invariance under smooth diffeomorphisms.
\begin{thm}[Bauer et. al.~\cite{bauer2016}]\label{chentsov}
 The Fisher information metric is the only Riemannian metric (up to a multiple) that is defined on the space of all smooth, positive densities on a compact 
manifold $M$ of dimension $2$ or higher, and is invariant under all diffeomorphisms from $M$ into itself (where diffeomorphisms are smooth maps with smooth inverse, so they are a special case of sufficient statistic).
\end{thm}
\begin{rem}\label{rem:cencov}
However, we will not have the same nice result for the $L^p$-Fisher metric. Later in Proposition~\ref{orliczinvariant}, we will see that the finite combinations of $L^p$-Fisher metrics fail to include all the possible Finsler metrics invariant under all diffeomorphisms, and in fact the set of Luxemburg norms form a class of $C^1$-Finsler metrics invariant under all diffeomorphisms. 
\end{rem}

 \section{The case of  $\mathbb R$: real counterpart of Teichm\"{u}ller theory}\label{caseRsec} 
From the preceding discussion, we see a concise and clean treatment of the $L^p$-geometry on $\operatorname{Dens}(M)$  and $\operatorname{Prob}(M)$ of a compact manifold $M$. It is natural to consider the case for non-compact manifolds, and it is indeed important, since the theory of compact manifolds fails to include the case $M=\mathbb R$ and the corresponding normal distributions and exponential distributions.
However, now the spaces $\operatorname{Dens}(\mathbb R)$ and  $\operatorname{Prob}(\mathbb R)$ become too large to carry out analysis in these spaces, as the metric defined by integration on $M$ easily diverges for some tangents. Therefore, we restrict our analysis to specific spaces of densities (probability densities) with decay conditions, which can be viewed as certain classes of diffeomorphism group acting on elements of $\operatorname{Prob}(\mathbb R)$.

\subsection{The structure of $\Diff_{-\infty}(\mathbb R)$} Note that the group of all smooth, orientation preserving diffeomorphisms of the real line is not an open subset of the space $C^{\infty}(\mathbb R,\mathbb R)$ and consequently it is not a smooth Fr\'echet manifold.
So we need to restrict ourselves to groups of diffeomorphisms with certain decay conditions, which allows us to regain a manifold structure in these groups; see e.g., ~\cite{michor2013zoo,kriegl2015exotic} for further information. In the following, we consider the diffeomorphism group
$\Diff_{-\infty}(\mathbb R)$
related to the Sobolev space $W^{\infty,1}(\mathbb R)$, which is introduced in~\cite{bauer2014homogeneous}.
\begin{align*}
\Diff_{-\infty}(\mathbb R)=\left\{\varphi=\id+f: f'\in W^{\infty,1}(\mathbb R),\; f'>-1,\text{ and } \lim_{x\to -\infty} f(x)=0 \right\},    
\end{align*}
where $W^{\infty,1}(\mathbb R)={\bigcap} W^{k,1}(\mathbb R)$ is defined as the intersection of all Sobolev spaces of order $k\geq 0$. It has been shown in~\cite{bauer2014homogeneous} that this space is a smooth Fr\'echet Lie-groups with Lie-algebra:
\begin{align*}
\mathfrak g_{-\infty}=\left\{u: u'\in W^{\infty,1}(\mathbb R)\text{ and } \lim_{x\to -\infty} u(x)=0 \right\}.
\end{align*}
Bauer et. al. show that there is an isometry mapping the diffeomorphism group $\Diff_{-\infty}(\mathbb R)$ with the right-invariant $\dot W^{1,p}$-Finsler metric to an open subset of a vector space.
\begin{thm}[Bauer et. al.~\cite{bauer2021paper}]\label{thm:isometry:nonperiodic}
For $p\in [1,\infty)$, the mapping
\begin{equation}
\Phi_p: \begin{cases}
\left(\Diff_{-\infty}(\mathbb R), \dot W^{1,p}\right)  	&\to  \left(W^{\infty,1}(\mathbb R),L^p\right)\\ 
\varphi 						&\mapsto p\left(\varphi'^{1/p}-1\right)
\end{cases}
\end{equation}
is an isometric embedding. Furthermore, the image $\mathcal U=\Phi_p(\Diff_{-\infty}(\mathbb R))$  is 
the set of all positive functions in $W^{\infty,1}(\mathbb R)$, i.e.,
\begin{equation}\label{eq:nonperiodic:U}
\mathcal U=\{f \in W^{\infty,1}(\mathbb R):f>-p\}.
\end{equation}
The inverse of $\Phi$ is given by
\begin{equation}\label{eq:Phiinverse:nonperiodic}
\Phi_p^{-1}: \begin{cases}
W^{\infty,1}(\mathbb R)  &\to  \Diff_{-\infty}(\mathbb R) \\ f &\mapsto x+\int_{-\infty}^x \left( \left(\frac{f(\tilde x)}{p}+1\right)^p-1\right)d\tilde x.
\end{cases}
\end{equation}
 \end{thm}
 We can go one step forward by considering the case for $p\to\infty$.
\begin{cor}\label{inftyisom}
the mapping
\begin{equation}
\Phi_\infty: \begin{cases}
\left(\Diff_{-\infty}(\mathbb R), \dot W^{1,\infty}\right)  	&\to  \left(W^{\infty,1}(\mathbb R),L^\infty\right)\\ 
\varphi 						&\mapsto \operatorname{log}(\varphi')
\end{cases}
\end{equation}
is an isometric diffeomorphism. Furthermore, $\Phi_\infty$ is the limit of the mappings $\Phi_p$ for $p\to\infty$. The inverse of $\Phi_\infty$ is given by
\begin{equation}\label{phiinverse}
\Phi^{-1}_\infty: \begin{cases}
W^{\infty,1}(\mathbb R)  &\to  \Diff_{-\infty}(\mathbb R) \\ f &\mapsto x+\int_{-\infty}^x \left( e^{f(\tilde x)}-1\right)d\tilde x.
\end{cases}
\end{equation}  
\end{cor}
\begin{proof}
Since $\varphi' -1 \in W^{\infty,1}(\RR)$, $\operatorname{log}(\varphi')$ is also in $W^{\infty,1}(\RR)$ (for $\operatorname{log}(1+\alpha) < \alpha$). Similarly, it is easy to see that the image of the inverse is indeed in $\Diff_{-\infty}(\RR)$. The calculation of isometry is similar to the one carried out in the preceding theorem, so we omit the proof. Next $\Phi_\infty$ is onto, since $\varphi'>0$ for orientation preserving diffeomorphisms $\varphi$ and direct calculation shows the existence of the inverse. It remains to prove the statement on the limit. By the monotone convergence theorem, $\Phi_p^{-1}$ converges to $\Phi^{-1}_\infty$ in formula~\ref{phiinverse}. Then a straightforward calculation shows 
\begin{align*}  
e^{f(x)}=D_{f,f'}\Phi^{-1}_\infty=\varphi',  
\end{align*}
which concludes the proof.
\end{proof}
From this corollary, it is clear that $\Diff_{-\infty}(\mathbb R)$ is contractible, since $W^{\infty,1}(\RR)$ is contractible.
\begin{rem}
As shown in~\cite{bauer2021paper}, the analysis above holds for spaces
of lower regularity, e.g., the space of integrable bi-Lipschitz homeomorphisms:
\begin{multline}
\biLip^{1,1}_{-\infty}(\RR) := \Big\{ \varphi=\id+f:   \varphi \text{ is invertible, }  \varphi^{-1}=\id+g,\\ f,g \in \dot{W}^{1,1}(\RR)\cap W^{1,\infty}(\RR), \,\, \lim_{x\to -\infty} f(x) = 0\Big\}.
\end{multline}
Here $\dot W^{1,1}(\RR)$ is the space of functions with an integrable derivative. This space is a manifold and a topological group, in which composition from the right is smooth (known as a half Lie-group; see e.g.,~\cite{kriegl2015exotic}).
Its Lie-algebra is the space
\[
\mathfrak{lip}^{1,1}_{-\infty}(\RR) := \{ u \in W^{1,\infty}(\RR)\cap \dot{W}^{1,1}(\RR) ~:~ \lim_{x\to -\infty} u(x) =0 \},
\]
on which all the $\dot{W}^{1,p}$ Finsler norms are well defined and also the $p\to \infty$ limit make sense. 

This extension to lower regularity has important implications  when compared to the universal Teichm\"{u}ller space, which is determined by the set of quasiconformal homeomorphisms of $\CC$ fixing $0,1$ and $\infty$ that is holomorphic when restricted to the lower half plane $\mathbb H^*$. By Ahlfors-Beurling extension theorem, this set can be further reduced to that of $\RR$-quasisymmetric homeomorphisms $f|_{\RR}:\RR\to\RR$ fixing $0,1$ with a holomorphic homeomorphism on the $\mathbb H^*$; see e.g.,~\cite{hubbard2006}. The fact that $f$ is holomorphic on $\mathbb H^*$ forces much stronger regularity on the boundary map: $f|_{\RR}$ are bi-Lipschitz homeomorphisms; for an investigation of the regularity of universal Teichm\"{u}ller space modeled on a unit disk, one can refer to~\cite{ratiu2015}. Therefore, our discussion of diffeomorphism groups on $\RR$ is related to the geometry of universal Teichm\"{u}ller space in an unexpected way.
\end{rem}
For simplicity and to avoid potential problems of regularity, we will still work on the Fr\'echet Lie group $\Diff_{-\infty}(\mathbb R)$ in the following subsections.

\subsection{Relations to the Schwarzian derivative}
One natural way to introduce the Schwarzian derivative is to consider the function of two variables 
\begin{equation}\label{potentialinfty}
 V(y,z)=\operatorname{log}\left(\frac{\varphi(y)-\varphi(z)}{y-z}\right),   
\end{equation}
for any $\varphi\in \Diff_{-\infty}(\mathbb R)$;  see e.g.,~\cite{schiffer1966} for Riemann surface case. Note that $F$ is well-defined, since $\varphi'>0$ for orientation preserving diffeomorphisms $\varphi$. Then the diagonal elements $u_\varphi$ are exactly onto the image of  $\Phi_\infty$,
\begin{equation}
u_\varphi(x):=\Phi_\infty(\varphi(x))=\operatorname{log}(\varphi')=V(x,x).    
\end{equation}
The second mixed partial derivative of $F$ is given by
\begin{align*}
\frac{\partial^2V(y,z)}{\partial y\partial z} =\frac{\varphi'(y)\varphi'(z)}{(\varphi(y)-\varphi(z))^2}-\frac{1}{(y-z)^2},    
\end{align*}
and the Schwarzian derivative is exactly the diagonal elements of the mixed derivative,
\begin{equation}
S\{\varphi\}(x):=   6\frac{\partial^2V(y,z)}{\partial y\partial z} |_{(y,z)=(x,x)}. 
\end{equation}
The diagonals $u_\varphi$ serve as the potential for the Schwarzian derivative $S$, similar to the K\"{a}hler potential,
\begin{equation}
S\{\varphi,x\}:=S\{\varphi\}(x)dx\otimes dx=\partial_y\partial_z P |_{(y,z)=(x,x)}.   
\end{equation}
The Schwarzian derivative of the composition $\varphi\circ\phi$ is given by the chain rule 
\begin{equation}\label{chainrule}
 S\{\varphi\circ\phi\}= S\{\varphi\}\circ\phi\cdot \phi'^2+S\{\phi\}.
\end{equation}
Note that from the construction of the Schwarzian,  any M\"{o}bius transformation $\varphi$ has vanishing Schwarzian $S\{\varphi\}$, and then the composition $\varphi\circ\phi$ has the
same Schwarzian derivative as $\phi$. Consequently the Schwarzian can be viewed as a sort of metric on right cosets of M\"{o}bius transformations (linear mappings in this case) that measures of how much a transformation derivates from a M\"{o}bius transformation.

Surely we have a quantitative relation between $u_\varphi$ and $S\{\varphi\}(x)$:
\begin{equation}\label{schwrzexp}
S\{\varphi\}(x)=u_\varphi''(x)-\frac{1}{2}u_\varphi'(x)^2.    
\end{equation}
Since any order derivative of $u_\varphi$  remains in $W^{\infty,1}(\mathbb R)$, so does the Schwarzian derivative $S\{\varphi\}$. Consequently it  has a unique preimage $\varphi_{S\{\varphi\}}$ in $\Diff_{-\infty}(\mathbb R)$, which can be expressed by using formula~\ref{phiinverse},
\begin{equation}
\varphi_{S\{\varphi\}}(x)=x+\int_{-\infty}^x \left( e^{u_\varphi''(\tilde x)-\frac{1}{2}u_\varphi'(\tilde x)^2}-1\right)d\tilde x.    
\end{equation}
From formula~\ref{chainrule} and an induction, it is easily seen that all iterations of a function with negative (or positive) Schwarzian will remain negative (resp. positive). Then we have the following proposition of interest in one dimensional dynamics.
\begin{prop}\label{1ddynamics}
 If $S\{\varphi\}>0$ (resp. $S\{\varphi\}<0$), then the tangent
 ${\varphi^n_{S\{\varphi^m\}}}'>1$ (resp. ${\varphi^n_{S\{\varphi^m\}}}'<1$) for any positive integer $n$ and $m$.   
\end{prop}
\begin{proof}
Assume that $S\{\varphi\}>0$. By the chain rule of derivative $(\varphi^n)'$, we see that $(\varphi^n)'>1$ if $\varphi'>1$. Hence it suffices to show ${\varphi_{S\{\varphi^m\}}}'>1$, which is proved by the positivity of $S\{\varphi^m\}$.
\end{proof}

\subsection{The $L^p$-approximation for Schwarzian derivative}
From the preceding discussion, we see that the real Schwarzian derivative can be derived from the Schwarz potential $V$ defined in equation~\ref{potentialinfty}. Since it is related to the isometry $\Phi_\infty$ with respect to the $L^\infty$-metric, it is natural to consider the $L^p$-potential $V_p$ with respect to the $L^p$-metric
\begin{equation}
V_p(y,z):=p\left\{\left(\frac{\varphi(y)-\varphi(z)}{y-z}\right)^{1/p}-1\right\}.    
\end{equation}
Similar to the previous case, the restriction of $V_p$ to the diagonal is $\Phi_p$,
\begin{align*}
V_p(x,x)=\Phi_p(\varphi).    
\end{align*}
The $L^p$-Schwarzian is then given by
\begin{equation}
S_{p}\{\varphi\}(x):=   6\frac{\partial^2V_p(y,z)}{\partial y\partial z} |_{(y,z)=(x,x)}.   
\end{equation}
It turns out that the pointwise limit of $S_p$ is just the Schwarzian $S$.
\begin{prop}
 The $L^p$-Schwarzian $S_{p}\{\varphi\}$ has the following expression
 \begin{equation}
 S_{p}\{\varphi\}(x)=\left(\frac{3}{2p}(\frac{\varphi''}{\varphi'})^2+S\{\varphi\}(x)\right)\varphi'^{1/p}.    
 \end{equation}
 Therefore, the limit of $S_p$ is the Schwarzian $S$ when $p\to\infty$.
\end{prop}
\begin{proof}
It is enough to calculate the mixed derivative for
\begin{align*}
F(y,z)=\left(\frac{\varphi(y)-\varphi(z)}{y-z}\right)^{1/p}.   
\end{align*}
After taking log and subsequent calculations, we get
\begin{align*}
p\partial_y\partial_z F|_{(y,z)=(x,x)}&=pF\frac{\partial_yF}{F}\frac{\partial_zF}{F}|_{(y,z)=(x,x)}+\frac{1}{6}S\{\varphi\}(x)F
\\&=\frac{1}{p}(\frac{\varphi''}{2\varphi'})^2\varphi'^{1/p}+\frac{1}{6}S\{\varphi\}(x)\varphi'^{1/p}.
\end{align*}
Here we use another equality
\begin{align*}
\frac{\partial_zF}{F}|_{(y,z)=(x,x)}=\frac{\partial_yF}{F}|_{(y,z)=(x,x)}=\frac{1}{p}\frac{\varphi''}{2\varphi'}.   
\end{align*}
\end{proof}

The $L^p$-Schwarzian no longer vanishes for  M\"{o}bius transformations, and it satisfies a new chain rule.
\begin{prop}
The chain rule of $L^p$-Schwarzian is given by
\begin{align*}
S_p\{\varphi\circ\phi\}= S_p\{\varphi\}\circ\phi\cdot \phi'^{2+1/p}+S_p\{\phi\}\cdot(\varphi'\circ\phi)^{1/p}+\frac{3}{p}\frac{\varphi''\circ\phi}{\varphi'\circ\phi}\phi''(\varphi\circ\phi)'^{1/p}.     
\end{align*}   
\end{prop}
\begin{proof}
The only thing we need to do here is to proceed with the derivative by parts by using the following equality.
\begin{align*}
 \left(\frac{\varphi\circ\phi(y)-\varphi\circ\phi(z)}{y-z}\right)^{1/p}=\left(\frac{\varphi\circ\phi(y)-\varphi\circ\phi(z)}{\phi(y)-\phi(z)}\right)^{1/p}\left(\frac{\phi(y)-\phi(z)}{y-z}\right)^{1/p}.  
\end{align*}
It is also clear that the chain rule converges to that of Schwarzian for $p\to\infty$.  
\end{proof}
\begin{rem}
The $L^p$-approximation here reminds us of doing the same thing for Riemann surfaces. But it is much harder, or maybe impossible. Since then we need to consider the holomorphic homeomorphisms in the lower half plane $\mathbb H^*$. These holomorphic functions are determined by the Beltrami equations in the upper half plane $\mathbb H$. The crucial problem here is the Beltrami forms $\mu$ belong to the $L^\infty$ space with norm less than one, and this regularity cannot be changed to $L^p$ for the use of the mapping theorem.   
\end{rem}

\subsection{The real Bers imbedding}
To include all the orientation preserving linear mappings, we enlarge $\Diff_{-\infty}(\mathbb R)$ to a bigger Fr\'echet Lie group
\begin{align*}
\widetilde\Diff_{-\infty}(\mathbb R)=\left\{\varphi=a\cdot\id+b+f: (\varphi-b)/a\in\Diff_{-\infty}(\mathbb R), a>0,\ b\in\RR \right\}.    
\end{align*}
It is easy to see that $\Diff_{-\infty}(\mathbb R)$ is the quotient of $\widetilde\Diff_{-\infty}(\mathbb R)$ by the real affine group of $\RR$, i.e.,
\begin{equation}\label{quotentgroup}
\Diff_{-\infty}(\mathbb R)=\widetilde\Diff_{-\infty}(\mathbb R)/\operatorname{A}(\mathbb R),    
\end{equation}
where the real affine group $\operatorname{A}(\mathbb R)=\{\varphi: \varphi(x)=ax+b,\; a>0,  b\in\RR\}$ acts by right cosets.
Indeed, we have the following identity
\begin{align*}
a\cdot\id+b+f=(a\cdot\id+b)\circ(\id+f/a).    
\end{align*}
We also introduce the space of all positive vertical shifts of $W^{\infty,1}(\mathbb R)$,
\begin{align*}
W_+^{\infty,1}(\mathbb R)=  \left\{g+a:g\in W^{\infty,1}(\mathbb R), \ a>0\right\}.  
\end{align*}
Then $\Phi_\infty$ in Corollary~\ref{inftyisom} can be extended to $\widetilde\Diff_{-\infty}(\mathbb R)$, and is no longer an injective isometry, but rather a covering map. 
\begin{prop}
The mapping
\begin{equation}
\Phi_\infty: \begin{cases}
\left(\widetilde\Diff_{-\infty}(\mathbb R), \dot W^{1,\infty}\right)  	&\to  \left(W_+^{\infty,1}(\mathbb R),L^\infty\right)\\ 
\varphi 						&\mapsto \operatorname{log}(\varphi')
\end{cases}
\end{equation}
is a covering map with the fibre isomorphic to $\RR$. Furthermore, $\operatorname{\Phi_\infty}$ is the projection of a trivial bundle.   
\end{prop}
\begin{proof}
It suffices to characterize the fibre of $\Phi_\infty$. Note that the decomposition $(a,b,f)$ of $\varphi$ in $\widetilde\Diff_{-\infty}(\mathbb R)$ is unique, due to the condition on $f$. Then the preimage $(a,b,f)$ of $\operatorname{log}(a_0+f_0')$ should satisfy $f'=a_0-a+f_0'$. Therefore, we must have $a=a_0$ and $f=f_0$. 
\end{proof}
From the quotient relation~\ref{quotentgroup}, the Fr\'echet Lie group $\operatorname{Diff}_{-\infty}(\RR)$ can be compared to the universal Teichmuller space and it is natural to consider the corresponding Bers imbedding in this space. Firstly we present a lemma showing that this mapping is indeed an embedding.
\begin{lem}
The mapping $\beta$ defined below is an embedding, i.e., the differential $D\beta$ is an injective.     
\end{lem}
\begin{proof}
Let $f=\operatorname{log}(\varphi')\in W^{\infty,1}(\RR)$. The mapping $\beta$ and its differential can be expressed
as
\begin{align*}
\beta:\varphi\to f''-1/2f'^2, \; \;  D\beta:\delta\varphi\to \delta f''-f'\delta f',    
\end{align*}
where $\delta f$ is a variation of $f$. Hence the kernel of $D\beta$ can be read off
\begin{equation}\label{ode1}
 \delta f=C_0\int^x_{-\infty} e^{f(\tilde x)}d \tilde x+C_1, 
\end{equation}
for constants $C_0$ and $C_1$. The equality~\ref{ode1} cannot hold for nonzero constants since both $f$
and $\delta f$ are in $W^{\infty,1}(\RR)$. Thus the kernel $\delta\varphi$ vanishes as $\delta\varphi'$ vanishes.
\end{proof}

\begin{thm}
The real Bers imbedding   \begin{equation}
\beta: \begin{cases}
\left(\Diff_{-\infty}(\mathbb R), \dot W^{1,\infty}\right)  	&\to  \left(W^{\infty,1}(\mathbb R),L^\infty\right)\\ 
\varphi 						&\mapsto S\{\varphi\}
\end{cases}
\end{equation} 
is an injective mapping onto an unbounded non-open subset of $W^{\infty,1}(\mathbb R)$ given by $\operatorname{Im}(\beta)=\{f''-1/2f'^2: \ f\in W^{\infty,1}(\mathbb R)\}$. Here $S$ is the Schwarzian derivative. 
\end{thm}
\begin{proof}
First note that a left composition with real affine transformation will not change the Schwarzian, hence $\beta$ is well defined on $\Diff_{-\infty}(\mathbb R)$. The injectivity comes from the fact that the only possible transformation annuls Schwarzian is a M\"{o}bius transformation. The characterization of the image of $\beta$ is a consequence of  corollary~\ref{inftyisom} and formula~\ref{schwrzexp}. To show non-openness, observe that
\begin{align*}
0\geq\int f''(x)dx-1/2\int f'^2(x)dx=\int g(x)dx,    
\end{align*}
since $f\in W^{\infty,1}(\mathbb R)$ implies $\operatorname{lim}_{x\to\pm\infty}f'(x)=0$, and then the first term in the middle vanishes. Consider the neighborhood of $0\in W^{\infty,1}(\mathbb R)$, choose a bump function $\eta\in C_c^\infty(\RR)$ with integration $\int\eta=1$, then for sufficiently small $\epsilon>0$, $g_\epsilon=\epsilon\eta$ is in any neighborhood of $0$ and has integration $\int g_\epsilon=\epsilon>0$, so there is no preimage 
$f_\epsilon$ such that $S\{f_\epsilon\}=g_\epsilon$.

It remains to show the unboundness, which can be seen from the following simple example. Take another smooth bump function $h(x)$ supported in a compact set containing $0$ such that $h''(0)=1$ and $h'(0)=0$.
Then take $h_n(x)=nh(nx)$, we have
$h_n\in W^{\infty,1}(\mathbb R)$ and at $x=0$,
\begin{align*}
|h_n''(0)-1/2h'_n(0)^2|=n^3|h''(0)|\to\infty.    
\end{align*}
\end{proof}
\begin{rem}
We see that the openness and boundedness of the Bers embedding for the Riemann surface case is no longer true for the real line case. It is left to find the exact form of $\operatorname{Im}(\beta)$, which is related to solving a real Riccati equation for $W^{\infty,1}(\mathbb R)$ functions.
\end{rem}

\subsection{A Fr\'echet submanifold of $\operatorname{Dens}(\mathbb R)$}
We consider a subgroup of 
$\Diff_{-\infty}(\mathbb R)$,
\begin{align*}
\widehat\Diff_{-\infty}(\mathbb R)=\left\{\varphi=\id+f: f'\in W^{\infty,1}(\mathbb R),\; f'>0,\text{ and } \lim_{x\to -\infty} f(x)=0 \right\}.    
\end{align*}
Note that $\widehat\Diff_{-\infty}(\mathbb R)$ is indeed a subgroup of $\Diff_{-\infty}(\mathbb R)$, since $(\varphi\circ\phi)'=\varphi'\circ\phi\cdot\phi'$ and $\varphi'>1,\phi'>1$ lead to $(\varphi\circ\phi)'>1$. It is obvious that $\widehat\Diff_{-\infty}(\mathbb R)$ is an open subset of $\Diff_{-\infty}(\mathbb R)$, and consequently it inherits the  Fr\'echet Lie group structure.  

There exists a one-to-one correspondence between $\widehat\Diff_{-\infty}(\mathbb R)$ and the space of all positive densities with density functions in $W^{\infty,1}(\mathbb R)$, denoted by
$\operatorname{Dens}^{\infty,1}(\mathbb R)$,
\begin{align*}
\operatorname{Dens}^{\infty,1}(\mathbb R)=\left\{\mu=g(x)dx: g\in W^{\infty,1}(\mathbb R),\; g>0\right\}.    
\end{align*}
This correspondence can be elevated to an isometry if we further endow $\operatorname{Dens}^{\infty,1}(\mathbb R)$ a $L^p$-metric $F_p$.
\begin{defi}
For $p\geq 1$, the $L^p$-metric $F_p$ on $\operatorname{Dens}^{\infty,1}(\mathbb R)$ is given by
\begin{equation}\label{FishermetricR}
F_p({\mu},a):=\left(\int|\frac{a}{\mu+dx}|^p(\mu+dx)\right)^{1/p},    
\end{equation}
 where $a\in T_\mu\operatorname{Dens}(M)$.
 The $L^p$-metric $F_\infty$ on $\operatorname{Dens}^{\infty,1}(\mathbb R)$ is given by
\begin{equation}\label{FishermetricRinfty}
F_\infty({\mu},a):=\operatorname{sup}_{x\in\RR}|\frac{a}{\mu+dx}|.    
\end{equation}
\end{defi}
\begin{rem}
We introduce the non-right invariant Finslerian metric $F_p$ because the standard $L^p$-Fisher-Rao metric easily fails in this noncompact case. For example, if we take the point $\mu=e^{-x^4}dx$ and the tangent $a=e^{-x^2}dx$, then the integrand $|a/\mu|^2\mu=e^{x^4-2x^2}$ does not converge in $\mathbb R$.    
\end{rem}

\begin{prop}\label{diffdensiso}
the mapping
\begin{equation}
\Gamma: \begin{cases}
\left(\widehat\Diff_{-\infty}(\mathbb R), \dot W^{1,p}\right)  	&\to  \left(\operatorname{Dens}^{\infty,1}(\mathbb R), F_p\right)\\ 
\varphi 						&\mapsto (\varphi-\operatorname{id})^*dx
\end{cases}
\end{equation}
is an isometric diffeomorphism. Therefore, $\operatorname{Dens}^{\infty,1}(\mathbb R)$ is a Fr\'echet manifold endowed with the metric $F_p$.  
\end{prop}
\begin{proof}
 First we have $(\varphi-\operatorname{id})^*dx=f'(x)dx$. The statement that $\Gamma$ is a diffeomorphism is clear except for the injectivity, which is guaranteed by the left decay condition of $f$.  It remains to show the isometry, which can be calculated as follows
 \begin{align*}
\|D_{\varphi,h}\Gamma\|^p_{F_p}= 
\int_{\mathbb R} {\varphi_x}^{1-p} |h_x|^p dx=||h||_{\dot W^{1,p}}.
\end{align*}
Finally, $\operatorname{Dens}^{\infty,1}(\mathbb R)$ is a Fr\'echet manifold since $\widehat\Diff_{-\infty}(\mathbb R)$ is a Fr\'echet Lie group.
\end{proof}
One of the important Fr\'echet submanifolds of $\operatorname{Dens}^{\infty,1}(\mathbb R)$ is the subspace of all probability densities 
\begin{align*}
\operatorname{Prob}^{\infty,1}(\mathbb R)=\left\{\mu=g(x)dx: g\in W^{\infty,1}(\mathbb R),\; g>0 ,\;\int\mu=1\right\}.    
\end{align*}

Parallel to the case of compact manifolds, we have an isometry from 
$\operatorname{Dens}^{\infty,1}(\mathbb R)$ with the $F_p$-metric to $W^{\infty,1}(\mathbb R)$
with the $L^p$-metric.

\begin{thm}\label{lpfisherisometryR}
The mapping 
\begin{align*}
\Psi_p: \begin{cases}
\left(\operatorname{Dens}^{\infty,1}(\mathbb R), F_p\right)  &\to  \left(W^{\infty,1}(\mathbb R),L^p\right)\\ \mu &\mapsto p\left((\frac{\mu+dx}{dx})^{1/p}-1\right)
\end{cases}
\end{align*}
is an isometric embedding.
Furthermore, the image $\mathcal U=\Psi_p(\operatorname{Dens}(M))$  is 
the set of all positive functions in $C^{\infty}(M)$, i.e.,
\begin{align*}
\mathcal U=\{f\in W^{\infty,1}(\mathbb R):\ f>-p\}.
\end{align*}
In particular, the restriction map $\Psi_p|_{\operatorname{Prob}(M)}$ is also an isometric embedding, and the image $\mathcal U'=\Psi_p(\operatorname{Prob}(M))$  is 
the set of all positive functions in a hypersurface of $W^{\infty,1}(\mathbb R)$, i.e.,
\begin{align*}
\mathcal U'=\{f\in W^{\infty,1}(\mathbb R):\ f>0,\ \int\left((\frac{f(x)}{p}+1)^p-1\right)dx=1\}.
\end{align*}
\end{thm}
\begin{proof}
 The proof is a combination of Theorem~\ref{lpfisherisometry} and Theorem~\ref{thm:isometry:nonperiodic}, so we omit it.
 \end{proof}
For $p\to\infty$, we have the following corollary,
\begin{cor}\label{lpfisherisometryR}
The mapping 
\begin{align*}
\Psi_\infty: \begin{cases}
\left(\operatorname{Dens}^{\infty,1}(\mathbb R), F_\infty\right)  &\to  \left(W^{\infty,1}(\mathbb R),L^\infty\right)\\ \mu &\mapsto \operatorname{log}(1+\frac{\mu}{dx})
\end{cases}
\end{align*}
is an isometric diffeomorphism.
In particular, the restriction map $\Psi_\infty|_{\operatorname{Prob}(M)}$ is also an isometric embedding, and the image $\mathcal U'=\Psi_\infty(\operatorname{Prob}(M))$  is 
the set of all positive functions in a  hypersurface of $W^{\infty,1}(\mathbb R)$, i.e.,
\begin{align*}
\mathcal U'=\{f\in W^{\infty,1}(\mathbb R):\ f>0,\ \int\left(e^{f(x)}-1\right)dx=1\}.
\end{align*}
\end{cor}
\begin{proof}
 Here the proof is a combination of Theorem~\ref{lpfisherisometry} and Theorem~\ref{inftyisom}.
 \end{proof}
Putting all of the above together, we have the following commutative diagram of isometric embeddings or diffeomorphisms for any $1\leq p \leq\infty$.

\begin{tikzcd}
 \left(\widehat\Diff_{-\infty}(\mathbb R), \dot W^{1,p}\right) \arrow{d}{\Gamma} \arrow{r}{\Phi_p}
 & \left(W^{\infty,1}(\mathbb R),L^p\right) \arrow{d}{\operatorname{id}} \\
 \left(\operatorname{Dens}^{\infty,1}(\mathbb R), F_p\right) \arrow{r}{\Psi_p}
 & \left(W^{\infty,1}(\mathbb R),L^p\right).
 \end{tikzcd}   
 
This diagram is essential as it gives us the correspondence between the space of diffeomorphisms, densities and the space of functions.
 
\begin{rem}[The higher dimensional case]
In the paper~\cite{michor2013zoo}, Michor and Mumford have shown that $\operatorname{Diff}_{H^\infty}(\mathbb\RR^n)$ is a Fr\'echet Lie group, and their method can be used to treat the case 
$\operatorname{Diff}_{W^{\infty,1}}(\mathbb\RR^n)$. But now the situation has become much more complicated, since there is no Newton-Leibniz theorem to retrieve the original function from its derivative. Moreover, one cannot find any direct condition on $\operatorname{det}(\mathbb I_n+df)$ that is equivalent to $\operatorname{det}(df)>0$, so the isometry in Proposition~\ref{diffdensiso} fails.
\end{rem}

\subsection{A remark on hyperbolic structures}
The universal Teichmuller space $\mathcal T(1)$ is the quotient of the set $\operatorname{QC}_0(\mathbb H)$ of quasiconformal mappings
of the upper half plane $\mathbb H$ fixing $0,1$ and $\infty$  by the group $\operatorname{PSL}(2,\RR)$ of Mobi\"{u}s transformations, i.e., 
\begin{align*}
\mathcal T(1)=\operatorname{QC}_0(\mathbb H) /\operatorname{PSL}(2,\RR).  
\end{align*}
It is well-known that the subgroup $\operatorname{PSL}(2,\RR)$ can be given a hyperbolic metric rendering it isometric to the unit tangent bundle $T^1\mathbb H$ endowed with the Sasaki metric.
In preceding discussions, we have observed that the Fr\'echet Lie group $\operatorname{Diff}_{-\infty}(\RR)$ is similar to $\mathcal T(1)$ in many aspects. Using this analogy, it is plain to imagine that the covering space $\widetilde\Diff_{-\infty}(\mathbb R)$ is rich in hyperbolic structures. In fact, a regular symmetric location-scale family of  probability densities   forms an upper half plane $\mathbb H$ with standard hyperbolic metric induced by the Fisher information metric.

Before stating this result, we first clarify the relation between this family and our Fr\'echet Lie groups. We select an open subgroup $G(\RR)$ of $\widetilde\Diff_{-\infty}(\mathbb R)$,
\begin{align*}
G(\mathbb R):=\left\{\varphi=a\cdot\id+b+f: (\varphi-b)/a\in\widehat\Diff_{-\infty}(\mathbb R), a>0,\ b\in\RR \right\}.
\end{align*}
Clearly we have a quotient relation similar to formula~\ref{quotentgroup}
\begin{equation}
 \widehat\Diff_{-\infty}(\mathbb R)=G(\mathbb R)/A(\RR),   
\end{equation}
where $A(\RR)$ is the affine group acting by right cosets. The isometric diffeomorphism $\Gamma$ defined in Proposition~\ref{diffdensiso} identifies the Lie group $\widehat\Diff_{-\infty}(\mathbb R)$ with 
the space  $\operatorname{Dens}^{\infty,1}(\mathbb R)$, consequently $G(\RR)$ can be decomposed into two parts, one is the contractible space $\operatorname{Dens}^{\infty,1}(\mathbb R)$, the other is the orientation preserving affine group $A(\RR)$. There is a canonical hyperbolic metric on $A(\RR)$, since it is naturally an upper half plane. In particular, for the regular symmetric  location-scale family in $\operatorname{Prob}^{\infty,1}(\mathbb R)$, the Fisher information metric induces a hyperbolic metric.
\begin{thm}[Amari~\cite{Amari2000}]
For regular symmetric  location-scale family of probability densities, i.e., all smooth probability density functions $g(-x)=g(x)$ such that the Fisher information metric exists, the Fisher information metric on the location-scale family $g(x;t,\sigma)$  is a  hyperbolic metric on $\mathbb H$ spanned by $\sigma$ and $t$. Here the family of probability densities $g(x;t,\sigma)$ is given by
\begin{align*}
 g(x;t,\sigma)=\frac{1}{\sigma}g(\frac{x-t}{\sigma}),   
\end{align*}
for $\sigma>0$ and $t\in\RR$.
\end{thm}
\begin{proof}
The result is well-known and the proof is a direct calculation, so we omit the proof.   
\end{proof}
%We can generalize this structure for any probability densities $g\in\operatorname{Prob}^{\infty,1}(\mathbb R)$. Let $z=t+\sigma i\in\mathbb H$. Consider the space $\mathcal L$ of all location-scale families  $g(x;z)$ for $z\in\mathbb H$ and its tangent $h(x;z)$ and $k(x;z)$ for $h,k\in T_g\operatorname{Prob}^{\infty,1}(\mathbb R)$. 

\section{ $L^p$-geometry on the space of symplectic forms}\label{symplecticsec}
Let $M$ be a closed smooth manifold of dimension $2n$,  with a potential symplectic form. A symplectic form on $M$ is a closed non-degenerate differential $2$-form. Now we consider the set of all possible symplectic forms $\omega$ on $M$. 

\subsection{The principal bundle structure on the space of symplectic forms} 
Parallel to the case of space of positive densities $\operatorname{Dens}(M)$ and space of probability densities $\operatorname{Prob}(M)$, there are two different spaces of symplectic forms to study on a closed manifold $M$.  
\begin{defi}
Let $\tilde S^+_M$ be the set of all possible symplectic forms imposed on the manifold $M$ with positive volume form. Denote by $S_M$ its subset of all symplectic forms 
$\omega$
with total volume $\int\omega^n/n!=1$.   
\end{defi}
$\tilde S^+_M$ can be endowed with a fibre bundle structure by introducing the canonical map from $\tilde S^+_M$ to 
$\operatorname{Dens}(M)$.
\begin{lem}
The surjective submersion $\pi:\tilde S^+_M\to \operatorname{Dens}(M)$ via:  $\omega\to\omega^n$ restricts to a submersion of $S_M$ onto its image $\operatorname{Prob}(M)$. 
The kernel of $\pi$ at each positive density $\mu$ is isomorphic.
\end{lem}
\begin{proof}
It suffices to show that for any positive density $\mu$, there exists a symplectic form $\omega$
such that $\omega^n=\mu$.
With the assumption on $M$, we have a symplectic form $\omega_0$ on $M$. Hence a positive multiple $k$ makes the following equality hold, $\int \mu=k^n\int \omega_0^n$. Using Moser's trick, we have $\varphi^*(k\omega_0)^n=\mu$ for some orientation-preserving diffeomorphism $\varphi$. The $2$-form $\varphi^*(k\omega_0)$ is obviously nondegenerate and closed, so it is the desired symplectic form. In addition, any smooth path emanating from $\mu$ in $\operatorname{Dens}(M)$ (resp. $\operatorname{Prob}(M)$) is given by $\mu(t)=k(t)\varphi(t)^*\omega_0^n$ (resp.  $\mu(t)=\varphi(t)^*\omega_0^n$), which shows submersion of the maps.

For the second part, note that for each fibre $\mathcal F_\mu=\{ \omega\in S_M\ : \ \mu=\omega^n\}$, we have $\mathcal F_\mu\cong \mathcal F_{\varphi^*\mu}$. Then every fibre is isomorphic to one another by using the Moser's trick once again. The case of $\tilde S^+_M$ differs only by rescaling by a positive number. 
\end{proof}
The diffeomorphism group  $\operatorname{Diff}_\mu(M)$ forms a closed subgroup of $\operatorname{Diff}(M)$ with the quotient $\operatorname{Diff}(M)/\operatorname{Diff}_\mu(M)\cong\operatorname{Prob}(M)$. More precisely, the projection onto the quotient space $\operatorname{Diff}(M)/\operatorname{Diff}_\mu(M)$ defines a smooth ILH
principal bundle with fibre $G=\operatorname{Diff}_\mu(M)$;  see e.g.,~\cite{khesin2011}. We can define its associated fibre bundle $E$ with fibre $\mathcal F\cong \mathcal F_\mu$ by
\begin{align*}
E:=\operatorname{Diff}(M)\times_{G} \mathcal F=(\operatorname{Diff}(M)\times \mathcal F)/G,    
\end{align*}
where the group action $G$ acts on $\operatorname{Diff}(M)\times \mathcal F$ by 
\begin{align*}
g\cdot(\varphi,\omega):=(g\circ\varphi,(g^{-1})^*\omega).
\end{align*}
It turns out that the ILH associated bundle $E$ is exactly the space of symplectic forms with constant volume $S_M$. 
\begin{thm}
The smooth surjective map $\pi: S_M\to\operatorname{Prob}(M)$ is the projection map of an associated bundle for the ILH principal bundle $p:\operatorname{Diff}(M)\to\operatorname{Diff}(M)/G$.     
\end{thm}
\begin{proof}
It is sufficient to find a homeomorphism $\Psi:S_M\to E$. We define a map from any $\omega\in S_M$ to $(\varphi,(\varphi^{-1})^*\omega)\in E$ by the following construction.
For any $\omega$ in $S_M$ and a fixed volume form $\mu$, $\int\omega^n$ and $\int\mu$ equals, then $\varphi^*\omega^n=\mu$ for some diffeomorphism $\varphi$. This is well-defined, since if $\varphi^*\omega^n=\psi^*\omega^n$, then we have $\varphi=\eta\circ\psi$ for some $\eta\in G$, and the equivalent relation
\begin{align*}
(\varphi,(\varphi^{-1})^*\omega)= (\eta\circ\psi,(\eta^{-1})^*(\psi^{-1})^*\omega)\sim (\psi,(\psi^{-1})^*\omega).   
\end{align*}
The definition of the inverse of $\Psi$ is much more straightforward and clear to be well-defined, which is given by $\Phi:(\phi,\alpha)\to\phi^*\alpha$. It is not hard to check that $\Phi$ is indeed the inverse.
\end{proof}
\begin{rem}
In order to ensure the equivalence relation for the associated bundle $E$, $\Psi$ does not map every fibre to $\mathcal F_\mu$. But rather it maps the fibre $\mathcal F_{(\varphi^{-1})^*\mu}$ of $S_M$ to the fibre $\mathcal F_{(\varphi^{-2})^*\mu}$ of $E$.  
\end{rem}
\begin{cor}
The fibre $\mathcal F$ of the map $\pi$ is an ILH manifold, i.e., a Fr\'echet manifold.    
\end{cor}
The tangent space $T_{\omega_0}\mathcal F$ of the fibre $\mathcal F$ at symplectic form $\omega_0$
is the kernel of the Lefschetz operator $L^{n-1}$,
\begin{align*}
T_{\omega_0}\mathcal F=\{a\in\Omega^2_{\operatorname{closed}}(M): \ a\wedge\omega_0^{n-1}=0\}=\operatorname{ker}(L^{n-1}),    
\end{align*}
where $L:\Omega^k(M)\to\Omega^{k+2}(M)$, $L(\alpha)=\omega\wedge\alpha$. In particular for the case of K\"{a}hler manifold $(M,\omega_0)$, $T_{\omega_0}\mathcal F$ is exactly the set of all primitive closed $2$-forms in the sense of Lefschetz decomposition.

\subsection{The right invariant $L^p$-metric on the space of symplectic forms}
Now we can define the $L^p$-metric on $\tilde S^+_M$, and show that it is essentially the $L^p$-Fisher-Rao metric when restricted to $S_M$.
\begin{defi}
For any $\alpha,\beta\in T_{\omega_0}\tilde S^+_M$, the (degenerate) $L^p$-metric on $\tilde S^+_M$ is given by 
\begin{equation}\label{lpsymplectic}
F(\omega_0,\beta):= \left(\int_M\left|\frac{\beta\wedge\omega_0^{n-1}}{\omega_0^{n}}\right|^p\frac{\omega_0^{n}}{n!}\right)^{1/p}, 
\end{equation}

and in particular for $p=2$, we have the (degenerate) $L^2$-inner product
\begin{equation}\label{l2symplectic}
G_{\omega_0}(\alpha,\beta):= \int_M \left|\frac{\alpha\wedge\omega_0^{n-1}}{\omega_0^{n}}\frac{\beta\wedge\omega_0^{n-1}}{\omega_0^{n}}\right|\frac{\omega_0^{n}}{n!}.    
\end{equation}
\end{defi}

\begin{thm}
Consider the $L^p$-metric 
$F$ on $S_M$ and the $L^p$-Fisher-Rao metric on $\operatorname{Prob}(M)$.
The projection  $\pi:S_M\to\operatorname{Prob}(M)$  via $\pi:\omega\to\omega^n$ is a local isometry.  In addition, the $L^p$-metric $F$ on the orbit space of any $\omega_0\in S_M$ under action of $\operatorname{Diff}(M)$ is equivalent to the degenerate $\dot W^{1,p}$-metric on $\operatorname{Diff}(M)$.     
\end{thm}
\begin{proof}
The local isometry is clear when we take the derivative of the volume form $\omega^n$, which is $n\beta\wedge\omega^{n-1}$. 
 A smooth path of diffeomorphisms $\varphi(t)$ emanating from $\omega_0$ generates the curve of tangents $v(t):=\varphi_t\circ\varphi^{-1}$ in Eulerian coordinates. Its operation on the volume form is given by
 \begin{align*}
 \omega(t)^n=\varphi(t)^*\omega_0^n=\operatorname{Jac_{\omega_0^n}(\varphi(t))}\omega_0^n,    
 \end{align*}
 where $\operatorname{Jac}_{\omega_0^n}(\varphi)$ denotes the Jacobian of the mapping $\varphi$ 
 with respect to the volumn form
 $\omega_0^n$. Then by taking derivative of $t$,
we get
\begin{equation}
\partial_t\omega(t)^n=\partial_t\operatorname{Jac}_{\omega_0^n}\varphi(t)\omega_0^n=\operatorname{div}v(t)\circ\varphi(t)\cdot\operatorname{Jac}_{\omega_0^n}\varphi(t)\omega_0^n.   
\end{equation}
After taking $t=0$, we get the desired result.
\end{proof}

\begin{expl}
By Darboux's theorem, we can characterize the local form of the integrand in~\ref{lpsymplectic}, which is consistent with the preceding theorem. Indeed, within an open neighborhood $U$ of $M$, we can write the symplectic form $\omega_0$ as the standard form
$\omega_0=\sum dx_i\wedge dy_i$. Any $1$-form can then be expressed as $\sigma=\sum-f_i dx_i+g_i dy_i$, with $f_i$ and $g_i$ smooth. Therefore, 
\begin{align*}
\beta\wedge\omega^{n-1}_0=d\sigma\wedge\omega^{n-1}_0=\operatorname{div}(\sigma)\omega^{n}_0.    
\end{align*}
\end{expl}

According to isometry $\pi$, we see that the pullback metric $F_p$ from $L^p$-Fisher-Rao metric exactly measures the horizontal tangent bundle of $S_M$, and the vertical tangent bundle vanishes identically for this metric. Someone may raise the question of the necessity of such metrics.
The following \u{C}encov type theorem for $S_M$ gives us a satisfactory answer.
\begin{prop}
The inner product $G$ is the only Riemannian metric (up to a multiple)
on the horizontal tangents of $S_M$ that is 
invariant under all diffeomorphisms from $M$  to itself.  
\end{prop}
\begin{proof}
This is a straightforward consequence of Theorem~\ref{chentsov}.   Indeed, the map $\pi:S_M\to\operatorname{Prob}(M)$ is a surjective submersion. Then the invariance of a Riemannian metric on  
$S_M$
descends to the one on $\operatorname{Prob}(M)$. Hence the Riemannian metric is the pullback of the Fisher-Rao information metric.
\end{proof}
\begin{rem}
The question remains  whether there is a nontrivial metric on the fibre $F$ invariant under all diffeomorphisms, and if it exists whether it is unique by rescaling. Then together with the preceding theorem, we are able to determine all the Riemannian metrics invariant on $S_M$.    
\end{rem}
\subsection{Hodge decomposition of the space of symplectic forms}
%In this subsection, we restrict our scope to a closed K\"{a}hler manifold$(M, \omega_K)$. A K\"{a}hler manifold is a symplectic manifold $(M,\omega_K)$ equipped with an integrable compatible almost complex structure $J$. This implies that there is a positive definite inner product on the tangent space of $M$, i.e., $g(u,v)=\omega_K(u,Jv)$ for any $u,v\in TM$. Then the situation becomes much more simplified since now the Riemannian manifold structure $g$ on $M$ allows us to use the Hodge decomposition.
The standard fact that every compact manifold $M$ can be endowed with a Riemmanian metric $g$ allows us to use the Hodge decomposition. 
\begin{thm}
$\tilde S^+_M$ can be decomposed as the direct sum of $S_M$ and a real line,
$\tilde S^+_M=S_M\times\mathbb R^*$ (or $\tilde S^+_M=S_M\times\mathbb R^+$, depending on $n$).
$S_M$ is the disjoint union of finitely many pieces $A_\gamma$, with respect to each harmonic form $\gamma$ on $M$. Each piece $A_\gamma$ is a  submanifold of the quotient space $\Omega^1/\Omega^1_{\operatorname{closed}}$, where $\Omega^1_{\operatorname{closed}}$ is the set of all closed $1$-forms.   
\end{thm}
\begin{proof}
By Hodge decomposition with respect to the metric $g$, any symplectic form $\omega$ decomposes uniquely as
\begin{align*}
\omega=d\alpha+\gamma,
\end{align*}
where $\gamma$ is harmonic. The only part that makes a contribution to the total volume is $\gamma$, so the first assertion is obvious. Since the harmonic part $\gamma$ is the canonical representative of the de Rham cohomology class $H^2_{DR}(M,\mathbb R)$, there are only finitely many candidates for such $\gamma$. Therefore, $S_M$ consists of finitely many disjoint pieces $A_\gamma$ with respect to linearly independent harmonic forms $\gamma$, as the decomposition is unique and $\int\gamma^n=n!$. Each piece $A_\gamma$ can be characterized by the differential of $1$-forms $d\Omega^1\cong \Omega^1/\Omega^1_{\operatorname{closed}}$. For any $\alpha$ in $A_\gamma$, the corresponding $\omega$ is non-degenerate, and any small perturbation of $\alpha$ in the sense of $C^0$-topology will not change the non-degeneracy of $\omega$, thanks to the  compactness of $M$. In other words, $A_\gamma$ is a submanifold with a path-connected neighborhood.
\end{proof}
The next proposition characterizes the neighborhood of a symplectic form $\omega_0$ in $A_\gamma$ as a path-connected neighborhood of identity in $\operatorname{Diff}(M)$. 
\begin{prop}
Every symplectic form $\omega_0$ in $A_\gamma$ has a path-connected $C^0$-neighborhood in which it is of the form $\varphi^*\omega_0$ for some diffeomorphism $\varphi$. The converse is also true, there exists a $C^0$-neighborhood of identity in $\operatorname{Diff}(M)$ such that every $\varphi^*\omega_0$ and $\omega_0$ are path-connected.
\end{prop}
\begin{proof}
Consider any neighborhood  
$U$ of $\omega_0^n$ in $\pi(A_\gamma)$, then $\omega_0$ is contained in one of the path-connected components of $\pi^{-1}(U)$, namely $V$. Every $\omega$ in a path-connected component should remain in the same cohomology class, since the harmonic parts are discrete. Thus $V\subset A_\gamma$, and the first part of proof follows from Moser stability theorem. For the converse, consider the preimage $f^{-1}(V)$ of the following continuous mapping 
\begin{align*}
f:\operatorname{Diff}(M)\to A_\gamma \ \ \text{via}: \ \ \varphi\to\varphi^*\omega_0,  
\end{align*}
it is clear that the path connected component of $f^{-1}(V)$ containing identity is the desired neighborhood of identity.
\end{proof}

\section{Higher dimensional Gelfand-Fuchs cocycles} \label{gelfandsec}
The classical Virasoro group (or Bott-Virasoro group) is an infinite dimesional Lie group defined as the universal central extension of $\operatorname{Diff}(S^1)$. Its Lie algebra is generated by the Gelfand-Fuchs $2$-cocycle, and is strictly related to the KdV equation, which can be viewed as the Euler-Arnold equation of the geodesic flow on the Virasoro group~\cite{khesin1987}.
In this section, we generalize the classical Gelfand-Fuchs $2$-cocyles to a $n+1$-cocycle in the space of probability densities on any closed manifold $M$ of dimension $n$.

\subsection{The Bott-Thurston cocycle on $\operatorname{Diff}(M)$}
The Bott-Thurston cocycle was originally introduced in Bott's paper~\cite{bott1977} as a tool for characteristic classes on $\operatorname{Diff}(M)$ of orientation preserving diffeomorphisms.
\begin{defi}
Let $\mu$ be a volume form on $M$. For any $\varphi\in\operatorname{Diff}(M)$, we define $\alpha(\varphi)\in C^\infty(M,\mathbb R)$ as follows
\begin{align*}
\alpha(\varphi):=\frac{\varphi^*(\mu)}{\mu}.    
\end{align*}
Then the following $n+1$-cocycle $c$ is an element in the group cohomology $H^{n+1}  (\operatorname{Diff}(M), \mathbb R)$,
\begin{align*}
c(\varphi_1,\cdots,\varphi_{n+1}):&=\int_M \operatorname{log}\alpha(\tilde\varphi_1) \ d\operatorname{log}\alpha(\tilde\varphi_2)\wedge\cdots\wedge  d\operatorname{log}\alpha(\tilde\varphi_{n+1})
\\&=\int_M \operatorname{log}\operatorname{Jac}_\mu(\tilde\varphi_1) \ d\operatorname{log}\operatorname{Jac}_\mu(\tilde\varphi_2)\wedge\cdots\wedge  d\operatorname{log}\operatorname{Jac}_\mu(\tilde\varphi_{n+1}),
\end{align*}
where $\varphi_1,\cdots,\varphi_{n+1}\in\operatorname{Diff}(M)$, and $\tilde\varphi_i=\varphi_1\varphi_2\cdots\varphi_{i}$ for $1\leq i\leq n+1$.
\end{defi}
This construction has a great importance in infinite dimensional Lie group and dynamics. For the case $n=1$, i.e., $M$
is a circle, the Virasoro group is the set of pairs $(\varphi,a)\in\operatorname{Diff}(S^1)\times\mathbb R$ satisfying the multiplication law
\begin{align*}
 (\varphi,a)\circ(\psi,b)=\left(\varphi\circ\psi, \ a+b+c(\varphi,\psi) \right),   
\end{align*}
where the $2$-cocycle $c$ here is the Bott-Thurston cocycle defined above.

The Virasoro algebra on $S^1$ can be derived from the second mixed derivative of the Virasoro group. In the following Proposition~\ref{antisymomega}, we generalize this construction to that on any closed manifold $M$.
\begin{defi}\label{vir}
The Virasoro algebra is the vector space
$\operatorname{Vect(S^1)}\oplus\mathbb R$ equipped with the following commutation operation:
\begin{align*}
[(f(x)\partial_x,a),(g(x)\partial_x,b)]=\left((f'(x)g(x)-f(x)g'(x))\partial_x, \int f'(x)g''(x)dx\right),    
\end{align*}
for any two elements $(f(x)\partial_x,a)$ and $(g(x)\partial_x,b)$ in the Lie algebra.
\end{defi}

 The $n+1$-th mixed derivative of $c$ generates an antisymmetric form on the tangent bundle $T\operatorname{Diff}(M)$ strictly related to the Virasoro algebra. Since it vanishes for all the vectors in $T\operatorname{Diff}_\mu(M)$, it can be viewed as an antisymmetric form on $T\operatorname{Prob}(M)$.
\begin{prop}\label{antisymomega}
The below antisymmetric form $\omega$ on $T_\mu\operatorname{Prob}(M)$ is the mixed partial derivative of $c$,  \begin{align*}
\omega(a_1,\cdots,a_{n+1}):&= D^{(n+1)}_{\partial_1,\cdots\partial_{n+1}}c\ (a_1,\cdots,a_{n+1})\\&= \int_M \frac{a_1}{\mu}d\frac{a_2}{\mu}\wedge\cdots\wedge d\frac{a_{n+1}}{\mu},  
\end{align*} 
where each $a_i\in T_\mu\operatorname{Prob}(M)$ is the $i$-th tangent corresponding to the $i$-th component of diffeomorphisms $\varphi_i$, for $i=1,\cdots,n+1$. Moreover, $\omega$ is a $n+1$-cocycle in Lie algebra cohomology. In particular, for $n=1$, $\omega$ is the Gelfand-Fuchs $2$-cocycle.
\end{prop}

\begin{proof}
By using Moser's trick, any $\mu(t_i)$ emanating from $\mu$ with tangent vector $a_i$ can be realized by some diffeomorphism $\varphi_i^{t_i}$ such that ${\varphi_i^{t_i}}^*\mu=\mu(t_i)$. Then straightforward calculations show that
\begin{align*}
\left.\frac{\partial^n}{\partial t_1\cdots\partial t_{n+1}}\right|_{t_1=\cdots=t_{n+1}=0}c(\varphi_1^{t_1},\cdots,\varphi_{n+1}^{t_{n+1}})= \int_M \frac{a_1}{\mu}d\frac{a_2}{\mu}\wedge\cdots\wedge d\frac{a_{n+1}}{\mu}.  
\end{align*}
The antisymmetry of $\omega$ follows from the Stokes' theorem and the closedness of $M$. In Bott's paper~\cite{bott1977}, the first proposition informs us of the fact that $c$ is a $n+1$-cocycle satisfying the group cocycle condition, so its differential $\omega$ is also closed. For the case $M=S^1$, let $f(x)\partial_x$ and $g(x)\partial_x$ be two smooth vector fields on $S^1$, and $\mu=dx$. Then $f'(x)$ and $g'(x)$ can be realized as $a_i/\mu$, so we have the equality
\begin{align*}
\int \frac{a_1}{\mu}d\frac{a_2}{\mu}=\int f'(x)g''(x)dx, 
\end{align*}
which is exactly the Gelfand-Fuchs $2$-cocycle.
\end{proof}
\begin{rem}
Note that the mixed $n+1$-th partial derivative is the only non-vanishing $n+1$-th derivative on $c$. In other words, if the derivative was not done for some component $t_i$, then this term disappears as $\operatorname{log}(\mu/\mu)=0$.    
\end{rem}

\subsection{The explicit form of $n$-cocycle $\omega$ in $L^p$-sphere}
Taking advantage of the isometry between $\operatorname{Prob}(M)$ and the $L^p$-sphere, we are able to imbed the space of all thus defined $n+1$-cocycles $G_{n+1}$ into an open set of a function space endowed with the $L^p$-norm. 
\begin{thm}
The space of all thus defined $n+1$-cocycles $G_n$ are isometric to that on a open subset of  the $L^p$-sphere with the explicit form shown in~\ref{omegalp}. The antisymmetric form $\omega$ only differs by a multiple for various choices of $p$.
\end{thm}
\begin{proof}
Fix a volume form $dx$ on $M$.
 Remember in Proposition~\ref{lpfisherisometry}, the isometry to the subset of $L^p$-sphere is given by 
 \begin{align*}
 \mu\to f=p(\frac{\mu}{dx})^{1/p}.    
 \end{align*}
 The tangent vector $a$ at point $\mu$ mapped to the tangent $a f$ at $f$. The tangent vector $a$ operates on $f$ as 
 \begin{align*}
 a f=\frac{a}{dx}(\frac{\mu}{dx})^{1/p-1}= \frac{1}{p}\frac{a}{\mu}f  
 \end{align*}
 Then $\omega$ within the
 $L^p$-sphere can be expressed as 
\begin{equation}\label{omegalp}
\omega=p^{n+1}\int_M \frac{a_1 f}{f} d\frac{a_2 f}{f}\wedge\cdots\wedge d\frac{a_n f}{f}
\end{equation} 
\end{proof}
This theorem is essential since it reduces the complex calculations on group of diffeomorphisms finally to that on the positive part of the $L^p$-sphere. Moreover, it is quite surprising that $\omega$ is intrinsically invariant under the change of $L^p$-norm. Finally, we present a simple case $M=S^1$. 
\begin{expl}
When $M$ is a circle, any volume form $\mu=\varphi^*dx$ for the standard form $dx$ and a diffeomorphism $\varphi$. It can be calculated that $f=\varphi_x^{1/p}$ and $\frac{a_if}{f}=\frac{\varphi_{xt}}{\varphi_{x}}|_{t=0}=(\varphi_t\circ\varphi^{-1})_x$. So we see the smooth vector field here is $\varphi_t\circ\varphi^{-1}$, which is the Eulerian coordinate of $\varphi$.  
\end{expl}

\section{Finsler geometry on Orlicz spaces}\label{orliczsec}
In this section, we generalize the Proudman–Johnson equations in the Orlicz space setting.  This answers Problem 7 in P.~Giblisco~\cite{Gib20}. For this purpose,
we restrict ourselves to the smooth category, and further assume that the Young function $\Phi$ to be smooth at any nonzero point and that all the functions $f$ of Definition~\ref{orliczspace} are smooth. Under this condition, we can show that the Luxemburg norm in the Orlicz space gives rise to a $C^1$-Finsler structure with good geometric properties, such as invariance under diffeomorphisms. If the Young function $\Phi$ is invertible, we can then embed the space of probability densities $\operatorname{Prob}(M)\cong\operatorname{Diff}(M)/\operatorname{Diff}_\mu(M)$ into the sphere of the Orlicz space. In particular for $M$ being a circle, we can calculate the corresponding Euler-Arnold equation.

\subsection{The Finsler structure on the Orlicz space}
First we introduce the concept of an Orlicz space, and present basic properties.
\begin{defi}[Orlicz space]\label{orliczspace}
A Young function $\Phi$ is a symmetric convex function $\Phi: \mathbb{R}\to\mathbb{R}\cup\infty$ strictly increasing on $[0,\infty)$ such that $\Phi(0)=0$ and $\operatorname{lim}_{x\to\infty}x^{-1}\Phi(x)=+\infty$. Let $(X,\mathcal{B},\mu)$ be any measure space and $f : X\to\mathbb{R}$ a measurable function. The Luxemburg norm $||f||_{\Phi}$ is given by
\begin{equation}
||f||_{\Phi}:=\operatorname{inf}\left\{r>0: \int_{X} \Phi(\frac{f}{r})\mu\leq 1\right\}.
\end{equation}
The Orlicz space with respect to the Young function $\Phi$ is defined as
\begin{equation}
L^{\Phi}:=L^{\Phi}(\mu):=\left\{f \operatorname{measurable}: ||f||_{\Phi}<+\infty \right\}.
\end{equation}
\end{defi}

The Orlicz spaces are natural generalization of $L^p$ spaces. Note that when $\Phi(x)=|x|^p$ for $p\in[1,+\infty)$, we actually obtain the $L^p$-norm $\|f\|_{\Phi}=\|f\|_{L^p}$. The Orlicz space is a Banach space, i.e. a completed normed vector space;  see e.g.,~\cite{Ay2017}. Therefore, it is usually suitable for us to calculate variations and by using the Picard-Lindel\"{o}ff theorem for Banach spaces or Banach manifolds, the existence of local geodesic flow is guaranteed.

In the following we show that there is a  Finsler structure on any Orlicz space, to be more precisely,
a Finsler metric $F$ satisfying the positive homogeneity and  subadditivity of a Finsler structure, but which is not smooth in general. In the following we will refer to such an object as a Finsler structure with lower regularity, such as $C^1$.
\begin{lem}
There exists a natural $C^1$-Finsler structure on any Orlicz space $L_\Phi$ endowed with a Luxemburg norm $||\cdot||_\Phi$.    
\end{lem}
\begin{proof}
First note that any Orlicz space endowed with a Luxemburg norm is a Banach space, and in smooth category, the only possible function $f$ satisfying $||f||_\Phi=0$ have to be zero. Then it suffices to show the metric is $C^1$. To prove this,  
we calculate the variation of any smooth function $f$ with respect to  $||\cdot||_\Phi$. Let $s\cdot h$ be a small variation of $f$ for the direction $h$.
\begin{align*}
K(s)=K_0+K_1 s+K_2 s^2+\cdots:=||(f+s\cdot h)||_\Phi,
\end{align*}
where $K_0=||f||_\Phi$, and $K_i$ are the coefficients of the perturbation expansion for each $i \in \mathbb{N}^+$. 
Then we can derive the Taylor expansion of $\Phi(\frac{f+s\cdot h}{K(s)})$ 
\begin{align*}
\Phi(\frac{f+s\cdot h}{K(s)})&=\Phi(\frac{f+s\cdot h}{K_0+K_1 s+K_2 s^2+\cdots})\\&=\Phi(\frac{f}{K_0})+\frac{s}{K_0}\left(h-\frac{fK_1}{K_0}\right)\Phi'(\frac{f}{K_0})+\mathcal{O}(s^2)
\end{align*}
Under our conditions on $\Phi$, the infimum can be attained, hence
\begin{equation}\label{variation1}
\int\Phi(\frac{f+s\cdot h}{K(s)})=1.
\end{equation}
So we obtain the variation $K_1$ of the norm.  Since the coefficient of $s$ on the left hand side of the Equation ~\eqref{variation1} should be zero, we have 
\begin{equation}\label{firstderivative}
K_1=\frac{||f||_\Phi\int h\ \Phi'(\frac{f}{||f||_\Phi}) \ \mu}{\int f\ \Phi'(\frac{f}{||f||_\Phi}) \ \mu}.
\end{equation}
It remains to show that the first derivative $K_1$ has non-vanishing denominator for any $f\in T\operatorname{Prob}(M) \setminus\{0\}$. Indeed, $||f||_\Phi>0$ for any nonzero $f$, and the fact that $\Phi'$ is odd implies that the following integral is necessarily positive,
\begin{align*}
\int f\ \Phi'(\frac{f}{||f||_\Phi}) \ \mu>0.    
\end{align*}
%Since $f$ can be any nonzero smooth function in the tangent space, the integral in the denominator may vanish. Thus we see that the norm is not necessarily smooth, which shows that the Finsler structure is non-smooth.
\end{proof}
\begin{rem}
In general, this Finsler metric does not satisfy the strong convex assumption as can be shown for the $L^p$-case~\cite{Lu2023}. In addition, the non-smoothness of the norm implies that  calculations of geodesic equations below are only of a formal nature.
\end{rem}
The Finsler structure thus defined on the Orlicz space
inherits some nice geometric properties from $L^p$ spaces. For example, consider the Luxemburg norm $L^{\Phi}$ on $\operatorname{Prob}(M)$.
\begin{prop}\label{orliczinvariant}
The $C^1$-Finsler metric $F_\Phi$ on $\operatorname{Prob}(M)$  with respect to the Young function $\Phi$ is invariant under all diffeomorphisms from a compact manifold $M$ into itself. In particular, this is true for the Young function
\begin{align*}
\Phi(t)=|t|\operatorname{log}(1+|t|).
\end{align*}
\end{prop}

\begin{proof}
Let $F$ be the Finsler metric on $\operatorname{Prob}(M)$ with respect to the Young function $\Phi$. For any tangent $a\in T_\mu\operatorname{Prob}(M)$, any diffeomorphism $\varphi$ operates on $F$ as follows
\begin{align*}
 \varphi^*F(\mu,a)=F(\varphi^*\mu, \varphi^*a)&=\operatorname{inf}\left\{r>0: \int_{M} \Phi(\frac{|\varphi^*(a/\mu)|}{r})\varphi^*\mu\leq 1\right\}\\&=\operatorname{inf}\left\{r>0: \int_{M} \Phi(\frac{|a/\mu|}{r})\mu\leq 1\right\}\\&=F(\mu,a).   
\end{align*}
Therefore, the invariance of Finsler structure also holds for the Young function $\Phi(t)=|t|\operatorname{log}(1+|t|)$. 
\end{proof}
It is important to note that the $C^1$-smoothness is sufficient to include Finsler metrics on infinite-dimensional manifolds of interest, such as $\operatorname{Prob}(M)$ and $\operatorname{Dens}(M)$. This is because even for the simplest Finsler metric on $\operatorname{Prob}(M)$, i.e., the $L^p$-Fisher metric $F$, the Hessian of $F^2$ necessarily fails to exist everywhere~\cite{Lu2023}, rendering the failure of $C^2$-finsler metrics.

\subsection{$\operatorname{Prob}(M)$ embedded in the unit sphere of an Orlicz space} In the preceding subsection, we impose an  Orlicz space related $C^1$-Finsler metric on the tangent bundle $T\operatorname{Prob}(M)$. There is another way to directly relate the Orlicz space to $\operatorname{Prob}(M)$. Note that the Young function $\Phi$ is invertible on the positive real axis since it is strictly increasing. If $\rho$ is a density in $\operatorname{Dens}(M)$, then we denote the inverse $A^{\Phi}(\rho)=\Phi^{-1}(\rho)$ by the $\Phi$-embedding. The following proposition shows that any probability density can be embedded in the unit sphere of an Orlicz space by using $\Phi$-embedding.
\begin{prop}[P.~Giblisco~\cite{Gib20}]
Let the smooth Young function $\Phi$ be invertible on the positive real axis. For any $\rho\in\operatorname{Prob}(M)$, the inverse map $A^{\Phi}(\rho):=\Phi^{-1}(\rho)$ lies on the unit sphere of the Orlicz space.
\end{prop}
\begin{proof}
First by the fact the integral of $\rho$ is one, 
\begin{align*}
\int\Phi(\frac{|A^{\Phi}(\rho)|}{1})=\int\Phi(\Phi^{-1}(\rho))=1.    
\end{align*}
Thus we have $||A^{\Phi}(\rho)||_\Phi\leq1$. The equality holds since $\Phi$ is smooth and strictly increasing, so for any real $r<1$, we have 
\begin{align*}
\int\Phi(\frac{|A^{\Phi}(\rho)|}{r})>\int\Phi(\frac{|A^{\Phi}(\rho)|}{1})=1.   
\end{align*}
\end{proof}

\subsection{The Euler-Arnold equation of $\operatorname{Diff}(S^1)/\operatorname{Rot}(S^1)$ endowed with the right invariant Luxemburg norm} 
The existence of $\Phi$-embedding reminds us of  considering the geodesic flow of probability densities endowed with Luxemburg norm. Since there is no $q$-energy for such a norm, we have to take a length functional to derive a Euler-Arnold equation; see formula~\ref{length}. In $S^1$, any probability density $\rho dx$ can be realized as the pullback of the canonical volume form $dx$ under an orientation preserving diffeomorphism $\varphi\in\operatorname{Diff}(S^1)$, i.e., $\rho dx=\varphi^*dx=\varphi_xdx$. Therefore, the $L^p$ metric $F_p$ can be interpreted as the one on
$\operatorname{Diff}(S^1)$
\begin{align*}
F_p(\rho dx,\rho_tdx)=\left(\int_{S^1}|\frac{\rho_t}{\rho}|^p\rho\right)^{1/p}= \left(\int|(\varphi_t \circ\varphi^{-1})_x|^pdx\right)^{1/p}=||(\varphi_t \circ\varphi^{-1})_x||_{L^p}.   
\end{align*}
We can replace the $L^p$-norm by right invariant Luxemburg norm $\Phi$,
then we obtain
\begin{equation}\label{luxeburgeuler}
F_ {\Phi}(\varphi,\varphi_t)= ||(\varphi_t \circ\varphi^{-1})_x||_{\Phi}= ||\frac{\varphi_{tx}}{\varphi_{x}}\circ\varphi^{-1}||_{\Phi},
\end{equation}
where $\varphi_t \circ\varphi^{-1}$ is the Eulerian coordinate.
\begin{prop}
Assume that $\varphi_{tx}/\varphi_{x}\geq C>0$. Then the geodesic equation of the quotient space $\operatorname{Diff}({S^1})/ \operatorname{Diff}_\mu({S^1})$ with respect to the right invariant Luxemburg norm is given by  
\begin{equation}
\frac{K_{0}^2\Phi(\frac{f}{K_0})}{\int f\ \Phi'(\frac{f}{K_0}) \varphi_x\ dx}=\frac{d}{dt}\left(\frac{K_{0}\Phi'(\frac{f}{K_0})}{\int f\ \Phi'(\frac{f}{K_0}) \varphi_x\ dx}\right)+\frac{K_{0}\Phi'(\frac{f}{K_0})}{\int f\ \Phi'(\frac{f}{K_0}) \varphi_x\ dx}\frac{\varphi_{tx}}{\varphi_{x}}.
\end{equation}
\end{prop}
\begin{proof}
A geodesic flow for the norm $F$ is an extremal curve $\varphi(t,x)$ of the energy functional:
\begin{equation}
 E(\varphi):=\frac{1}{2}\int_0 ^1 F_ {\Phi}(\varphi,\varphi_t)\ dt.
\end{equation}
The corresponding energy functional reads as
\begin{equation} 
E(\varphi)=\frac{1}{2}\int_0^1 ||\frac{\varphi_{tx}}{\varphi_{x}}\circ\varphi^{-1}||_{\Phi}\ dt.
\end{equation}
In order to simplify the calculation, we assume that the function $f:=\varphi_{tx}/\varphi_{x}\geq C>0$ and that the infimum equals the minimum. From the definition of the Luxemburg norm, we have
\begin{align*}
||f\circ\varphi^{-1}||_{\Phi}&=\operatorname{inf}\left\{K>0: \int_{S^1} \Phi(\frac{f\circ\varphi^{-1}}{K})\ dx\leq 1\right\}\\&=\operatorname{inf}\left\{K>0: \int_{S^1} \Phi(\frac{f}{K}) \ \varphi_{x}\ dx\leq 1\right\}.
\end{align*}
First we consider the variation of the function $f$
\begin{align}
K(s)=K_0+K_1 s+K_2 s^2+\cdots:=||(f+s\cdot h)\circ\varphi^{-1}||.
\end{align}

Then we can derive the Taylor expansion of $\Phi(\frac{f+s\cdot h}{K(s)})$
\begin{align*}
\Phi(\frac{f+s\cdot h}{K(s)})&=\Phi(\frac{f+s\cdot h}{K_0+K_1 s+K_2 s^2+\cdots})\\&=\Phi(\frac{f}{K_0})+\frac{s}{K_0}\left(h-\frac{fK_1}{K_0}\right)\Phi'(\frac{f}{K_0})+\mathcal{O}(s^2)
\end{align*}
Since the infimum is assumed to be the minimum, we have the equality
\begin{equation}\label{variation}
\int\Phi(\frac{f+s\cdot h}{K(s)})(\varphi_x+s\cdot\delta \varphi_x)=1
\end{equation}
For the coefficient of $s$ on the left hand side of  equation ~\eqref{variation} should be zero, we get the variation $K_1$ of the norm  
\begin{equation}
K_1=\frac{K_{0}^2\int\Phi(\frac{f}{K_0})\delta \varphi_x\ dx+K_{0}\int h\ \Phi'(\frac{f}{K_0}) \varphi_x\ dx}{\int f\ \Phi'(\frac{f}{K_0}) \varphi_x\ dx}.
\end{equation}
Next we use the equalities
\begin{align*}
h=\delta f=\frac{\delta \varphi_{tx}}{\varphi_{x}}-\frac{\varphi_{tx}}{\varphi_{x}^2}\delta \varphi_{x},\ \
K_0=||\frac{\varphi_{tx}}{\varphi_{x}}\circ\varphi^{-1}||_{\Phi}
\end{align*}
and we have the variation of the energy functional
\begin{align*}
2\delta E(\varphi)(\delta \varphi) &=\int\frac{K_{0}^2\int\Phi(\frac{f}{K_0})\delta \varphi_x\ dx+K_{0}\int h\ \Phi'(\frac{f}{K_0}) \varphi_x\ dx}{\int f\ \Phi'(\frac{f}{K_0}) \varphi_x\ dx} dt
\\&=\int\int\frac{K_{0}^2\Phi(\frac{f}{K_0})}{\int f\ \Phi'(\frac{f}{K_0}) \varphi_x\ dx}\delta \varphi_x\ dx dt\\&-\int\int\frac{d}{dt}\left(\frac{K_{0}\Phi'(\frac{f}{K_0})}{\int f\ \Phi'(\frac{f}{K_0}) \varphi_x\ dx}\right)\delta \varphi_{x}\ dx dt\\&-\int\int\frac{K_{0}\Phi'(\frac{f}{K_0})}{\int f\ \Phi'(\frac{f}{K_0}) \varphi_x\ dx}\frac{\varphi_{tx}}{\varphi_{x}}\delta \varphi_{x}\ dx dt
\end{align*}
Thus we can read off the geodesic equation:
\begin{equation}
\frac{K_{0}^2\Phi(\frac{f}{K_0})}{\int f\ \Phi'(\frac{f}{K_0}) \varphi_x\ dx}=\frac{d}{dt}\left(\frac{K_{0}\Phi'(\frac{f}{K_0})}{\int f\ \Phi'(\frac{f}{K_0}) \varphi_x\ dx}\right)+\frac{K_{0}\Phi'(\frac{f}{K_0})}{\int f\ \Phi'(\frac{f}{K_0}) \varphi_x\ dx}\frac{\varphi_{tx}}{\varphi_{x}}.
\end{equation}
\end{proof}
\begin{rem}
 In particular for the $L^p$-case $\Psi(x)=|x|^p$, we have
\begin{align*}
\int f\ \Phi'(\frac{f}{K_0}) \varphi_x\ dx=\frac{p}{K_0^{p-1}}\int(\frac{\varphi_{tx}}{\varphi_{x}})^p \varphi_{x}\ dx=pK_0
\end{align*}
If we further restrict ourselves to paths of constant speed, the geodesic equation can be simplified to
\begin{equation}
K_{0}^2\Phi(\frac{f}{K_0})=\frac{d}{dt}\left(K_{0}\Phi'(\frac{f}{K_0})\right)+K_{0}\Phi'(\frac{f}{K_0})\frac{\varphi_{tx}}{\varphi_{x}}
\end{equation}
and then to
\begin{equation}
p\frac{d}{dt}(\frac{\varphi_{tx}}{\varphi_{x}})^{p-1}+(\frac{\varphi_{tx}}{\varphi_{x}})^p=0.
\end{equation} 
Note that this gives us exactly the geodesic equation as derived directly from the $L^p$-norm in Theorem~\ref{geodLPFR}.
\end{rem}

\bibliographystyle{abbrv}

\end{document}